\documentclass[12pt,oneside]{amsart}
\usepackage{amsmath,amssymb, amsthm,mathtools}
\usepackage[margin=1.5in]{geometry}
\usepackage[parfill]{parskip}
\usepackage{enumitem}
\usepackage{tikz}
\usepackage{tikz-cd} 
\usetikzlibrary{shapes}
\usepackage{dsfont}
\usepackage[caption=false]{subfig}
\usepackage{mathrsfs}
\usepackage{bbm}
\usepackage{wasysym, marvosym}
\usepackage{appendix}
\usepackage{microtype}

\DeclarePairedDelimiter{\floor}{\lfloor}{\rfloor}
\providecommand{\N}{\mathbb{N}}
\providecommand{\R}{\mathbb{R}}

\providecommand{\T}{\mathbb{T}}
\providecommand{\Z}{\mathbb{Z}}
\providecommand{\Q}{\mathbb{Q}}
\providecommand{\F}{\mathbb{F}}

\renewcommand{\vec}[1]{\boldsymbol{#1}}

\newcommand{\paren}[1]{\left( #1 \right)}
\newcommand{\brac}[1]{\left[ #1 \right]}

\newcommand{\abs}[1]{\left\vert#1\right\vert}

\newcommand{\set}[1]{\left\{#1\right\}}

\newcommand{\bhexagon}{\mathord{\raisebox{0.0pt}{\tikz{\node[draw,scale=.55,regular polygon, regular polygon sides=6,fill=black](){};}}}}

\DeclareMathOperator{\rank}{rank}
\DeclareMathOperator{\im}{im}

\newtheorem{Theorem}{Theorem}
\newtheorem{Lemma}[Theorem]{Lemma}
\newtheorem{Definition}[Theorem]{Definition}
\newtheorem{Proposition}[Theorem]{Proposition}

\newtheorem{Corollary}[Theorem]{Corollary}

\begin{document}

\title[Homological Percolation on a torus]{Homological percolation on a torus: plaquettes and permutohedra}

\author{Paul Duncan}
\email{paul.duncan@mail.huji.ac.il}
\thanks{P.D.\ gratefully acknowledges the support of NSF-DMS \#1547357.}
\address{Einstein Institute of Mathematics, Hebrew University of Jerusalem, Jerusalem 91904, Israel}
\author{Matthew Kahle}
\email{mkahle@math.osu.edu}
\address{Department of Mathematics, Ohio State University, Columbus, OH 43210}
\thanks{M.K.\ gratefully acknowledges the support of NSF-DMS \#2005630 and NSF-CCF \#1839358.}
\author{Benjamin Schweinhart}
\email{bschwei@gmu.edu}
\address{Department of Mathematical Sciences, George Mason University, Fairfax, VA 22030}
\thanks{B.S.\ was was supported in part by a NSF Mathematical Sciences Postdoctoral Research Fellowship
under award number NSF-DMS \#1606259.}

\begin{abstract}
We study higher-dimensional homological analogues of bond percolation on a square lattice and site percolation on a triangular lattice.

By taking a quotient of certain infinite cell complexes by growing sublattices, we obtain finite cell complexes with a high degree of symmetry and with the topology of the torus $\T^d$. When random subcomplexes induce nontrivial $i$-dimensional cycles in the homology of the ambient torus, we call such cycles \emph{giant}. We show that for every $i$ and $d$ there is a sharp transition from nonexistence of giant cycles to giant cycles spanning the homology of the torus.

We also prove convergence of the threshold function to a constant in certain cases. In particular, we prove that $p_c=1/2$ in the case of middle dimension $i=d/2$ for both models. This gives finite-volume high-dimensional analogues of Kesten's theorems that $p_c=1/2$ for bond percolation on a square lattice and site percolation on a triangular lattice.

\end{abstract}

\maketitle

\section{Introduction}

Various models of percolation are fundamental in statistical mechanics; classically, they study the emergence of a giant component in random structures. From early in the mathematical study of percolation, geometry and topology have been at the heart of the subject. Indeed, Frisch and Hammersley wrote in 1963~\cite{FH63} that, ``Nearly all extant percolation theory deals with regular interconnecting structures, for lack of knowledge of how to define randomly irregular structures. Adventurous readers may care to rectify this deficiency by pioneering branches of mathematics that might be called \emph{stochastic geometry} or \emph{statistical topology}.'' 

The main geometric structure of interest in percolation theory thus far is often the existence of an infinite component or infinite path. A path which wraps around the torus can be seen as a finite volume analogue of this infinite component. We study events defined in terms of geometric structures that are higher-dimensional generalizations  of such paths. The ``homological percolation'' property we consider is one that was recently studied by Bobrowski and Skraba~\cite{bobrowski2020homological,bobrowski2020homological2}. The possibility of studying homological percolation in a $2$-dimensional torus or surface of genus $g$ was discussed earlier in~\cite{LPSA94}, and~\cite{Pinson94},~\cite{MDSA09}. The more general setting of $i$-dimensional homological percolation in a $d$-dimensional torus was, to our knowledge, first carefully studied in~\cite{bobrowski2020homological} and~\cite{bobrowski2020homological2}. The setup is as follows.

\begin{figure}[t]
\centering

\begin{tikzpicture}[line width=.25mm,scale=0.4] 

\draw[dashed, line width=.1mm] (0,0)--(10,0)--(10,10)--(0,10)--cycle;

\draw[blue] ( 1/4 , 1/4 )--( 5/4 , 1/4 );
\draw[blue] ( 1/4 , 1/4 )--( 0 , 1/4 );
\draw[blue] ( 37/4 , 1/4 )--( 10 , 1/4 );
\draw[blue] ( 1/4 , 5/4 )--( 0 , 5/4 );
\draw[blue] ( 37/4 , 5/4 )--( 10 , 5/4 );
\draw[blue] ( 1/4 , 9/4 )--( 5/4 , 9/4 );
\draw[blue] ( 1/4 , 13/4 )--( 1/4 , 17/4 );
\draw[blue] ( 1/4 , 13/4 )--( 5/4 , 13/4 );
\draw[blue] ( 1/4 , 13/4 )--( 0 , 13/4 );
\draw[blue] ( 37/4 , 13/4 )--( 10 , 13/4 );
\draw[blue] ( 1/4 , 17/4 )--( 5/4 , 17/4 );
\draw[blue] ( 1/4 , 17/4 )--( 0 , 17/4 );
\draw[blue] ( 37/4 , 17/4 )--( 10 , 17/4 );
\draw[blue] ( 1/4 , 21/4 )--( 1/4 , 25/4 );
\draw[blue] ( 1/4 , 21/4 )--( 5/4 , 21/4 );
\draw[blue] ( 1/4 , 25/4 )--( 5/4 , 25/4 );
\draw[blue] ( 1/4 , 25/4 )--( 0 , 25/4 );
\draw[blue] ( 37/4 , 25/4 )--( 10 , 25/4 );
\draw[blue] ( 1/4 , 29/4 )--( 1/4 , 33/4 );
\draw[blue] ( 1/4 , 29/4 )--( 5/4 , 29/4 );
\draw[blue] ( 1/4 , 37/4 )--( 5/4 , 37/4 );
\draw[blue] ( 1/4 , 37/4 )--( 0 , 37/4 );
\draw[blue] ( 37/4 , 37/4 )--( 10 , 37/4 );
\draw[blue] ( 5/4 , 1/4 )--( 5/4 , 5/4 );
\draw[blue] ( 5/4 , 1/4 )--( 5/4 , 0 );
\draw[blue] ( 5/4 , 37/4 )--( 5/4 , 10 );
\draw[blue] ( 5/4 , 1/4 )--( 9/4 , 1/4 );
\draw[blue] ( 5/4 , 5/4 )--( 9/4 , 5/4 );
\draw[blue] ( 5/4 , 9/4 )--( 5/4 , 13/4 );
\draw[blue] ( 5/4 , 9/4 )--( 9/4 , 9/4 );
\draw[blue] ( 5/4 , 13/4 )--( 9/4 , 13/4 );
\draw[blue] ( 5/4 , 17/4 )--( 5/4 , 21/4 );
\draw[blue] ( 5/4 , 21/4 )--( 9/4 , 21/4 );
\draw[blue] ( 5/4 , 25/4 )--( 5/4 , 29/4 );
\draw[blue] ( 5/4 , 25/4 )--( 9/4 , 25/4 );
\draw[blue] ( 5/4 , 29/4 )--( 5/4 , 33/4 );
\draw[blue] ( 5/4 , 33/4 )--( 5/4 , 37/4 );
\draw[blue] ( 5/4 , 33/4 )--( 9/4 , 33/4 );
\draw[blue] ( 5/4 , 37/4 )--( 9/4 , 37/4 );
\draw[blue] ( 9/4 , 1/4 )--( 9/4 , 5/4 );
\draw[blue] ( 9/4 , 1/4 )--( 13/4 , 1/4 );
\draw[blue] ( 9/4 , 5/4 )--( 9/4 , 9/4 );
\draw[blue] ( 9/4 , 9/4 )--( 9/4 , 13/4 );
\draw[blue] ( 9/4 , 9/4 )--( 13/4 , 9/4 );
\draw[blue] ( 9/4 , 13/4 )--( 13/4 , 13/4 );
\draw[blue] ( 9/4 , 21/4 )--( 9/4 , 25/4 );
\draw[blue] ( 9/4 , 29/4 )--( 9/4 , 33/4 );
\draw[blue] ( 9/4 , 29/4 )--( 13/4 , 29/4 );
\draw[blue] ( 9/4 , 33/4 )--( 13/4 , 33/4 );
\draw[blue] ( 9/4 , 37/4 )--( 13/4 , 37/4 );
\draw[blue] ( 13/4 , 1/4 )--( 13/4 , 0 );
\draw[blue] ( 13/4 , 37/4 )--( 13/4 , 10 );
\draw[blue] ( 13/4 , 1/4 )--( 17/4 , 1/4 );
\draw[blue] ( 13/4 , 5/4 )--( 17/4 , 5/4 );
\draw[blue] ( 13/4 , 9/4 )--( 17/4 , 9/4 );
\draw[blue] ( 13/4 , 13/4 )--( 13/4 , 17/4 );
\draw[blue] ( 13/4 , 13/4 )--( 17/4 , 13/4 );
\draw[blue] ( 13/4 , 17/4 )--( 13/4 , 21/4 );
\draw[blue] ( 13/4 , 17/4 )--( 17/4 , 17/4 );
\draw[blue] ( 13/4 , 21/4 )--( 13/4 , 25/4 );
\draw[blue] ( 13/4 , 37/4 )--( 17/4 , 37/4 );
\draw[blue] ( 17/4 , 1/4 )--( 17/4 , 5/4 );
\draw[blue] ( 17/4 , 1/4 )--( 17/4 , 0 );
\draw[blue] ( 17/4 , 37/4 )--( 17/4 , 10 );
\draw[blue] ( 17/4 , 9/4 )--( 17/4 , 13/4 );
\draw[blue] ( 17/4 , 17/4 )--( 17/4 , 21/4 );
\draw[blue] ( 17/4 , 17/4 )--( 21/4 , 17/4 );
\draw[blue] ( 17/4 , 29/4 )--( 17/4 , 33/4 );
\draw[blue] ( 17/4 , 33/4 )--( 21/4 , 33/4 );
\draw[blue] ( 17/4 , 37/4 )--( 21/4 , 37/4 );
\draw[blue] ( 21/4 , 1/4 )--( 21/4 , 0 );
\draw[blue] ( 21/4 , 37/4 )--( 21/4 , 10 );
\draw[blue] ( 21/4 , 5/4 )--( 21/4 , 9/4 );
\draw[blue] ( 21/4 , 9/4 )--( 25/4 , 9/4 );
\draw[blue] ( 21/4 , 13/4 )--( 21/4 , 17/4 );
\draw[blue] ( 21/4 , 17/4 )--( 21/4 , 21/4 );
\draw[blue] ( 21/4 , 25/4 )--( 21/4 , 29/4 );
\draw[blue] ( 21/4 , 25/4 )--( 25/4 , 25/4 );
\draw[blue] ( 21/4 , 29/4 )--( 21/4 , 33/4 );
\draw[blue] ( 21/4 , 29/4 )--( 25/4 , 29/4 );
\draw[blue] ( 21/4 , 33/4 )--( 21/4 , 37/4 );
\draw[blue] ( 21/4 , 33/4 )--( 25/4 , 33/4 );
\draw[blue] ( 25/4 , 1/4 )--( 29/4 , 1/4 );
\draw[blue] ( 25/4 , 5/4 )--( 29/4 , 5/4 );
\draw[blue] ( 25/4 , 9/4 )--( 25/4 , 13/4 );
\draw[blue] ( 25/4 , 13/4 )--( 25/4 , 17/4 );
\draw[blue] ( 25/4 , 13/4 )--( 29/4 , 13/4 );
\draw[blue] ( 25/4 , 17/4 )--( 25/4 , 21/4 );
\draw[blue] ( 25/4 , 17/4 )--( 29/4 , 17/4 );
\draw[blue] ( 25/4 , 21/4 )--( 25/4 , 25/4 );
\draw[blue] ( 25/4 , 33/4 )--( 25/4 , 37/4 );
\draw[blue] ( 29/4 , 1/4 )--( 29/4 , 5/4 );
\draw[blue] ( 29/4 , 1/4 )--( 29/4 , 0 );
\draw[blue] ( 29/4 , 37/4 )--( 29/4 , 10 );
\draw[blue] ( 29/4 , 1/4 )--( 33/4 , 1/4 );
\draw[blue] ( 29/4 , 5/4 )--( 33/4 , 5/4 );
\draw[blue] ( 29/4 , 9/4 )--( 29/4 , 13/4 );
\draw[blue] ( 29/4 , 13/4 )--( 33/4 , 13/4 );
\draw[blue] ( 29/4 , 17/4 )--( 29/4 , 21/4 );
\draw[blue] ( 29/4 , 21/4 )--( 29/4 , 25/4 );
\draw[blue] ( 29/4 , 25/4 )--( 33/4 , 25/4 );
\draw[blue] ( 29/4 , 29/4 )--( 29/4 , 33/4 );
\draw[blue] ( 29/4 , 29/4 )--( 33/4 , 29/4 );
\draw[blue] ( 29/4 , 33/4 )--( 33/4 , 33/4 );
\draw[blue] ( 29/4 , 37/4 )--( 33/4 , 37/4 );
\draw[blue] ( 33/4 , 5/4 )--( 37/4 , 5/4 );
\draw[blue] ( 33/4 , 9/4 )--( 37/4 , 9/4 );
\draw[blue] ( 33/4 , 17/4 )--( 33/4 , 21/4 );
\draw[blue] ( 33/4 , 17/4 )--( 37/4 , 17/4 );
\draw[blue] ( 33/4 , 21/4 )--( 33/4 , 25/4 );
\draw[blue] ( 33/4 , 25/4 )--( 33/4 , 29/4 );
\draw[blue] ( 33/4 , 25/4 )--( 37/4 , 25/4 );
\draw[blue] ( 33/4 , 29/4 )--( 33/4 , 33/4 );
\draw[blue] ( 33/4 , 33/4 )--( 37/4 , 33/4 );
\draw[blue] ( 37/4 , 1/4 )--( 37/4 , 5/4 );
\draw[blue] ( 37/4 , 1/4 )--( 37/4 , 0 );
\draw[blue] ( 37/4 , 37/4 )--( 37/4 , 10 );
\draw[blue] ( 37/4 , 5/4 )--( 37/4 , 9/4 );
\draw[blue] ( 37/4 , 9/4 )--( 37/4 , 13/4 );
\draw[blue] ( 37/4 , 17/4 )--( 37/4 , 21/4 );
\draw[blue] ( 37/4 , 29/4 )--( 37/4 , 33/4 );

\draw[blue,line width=.75mm] (0,6.25)--++(1.25,0)--++(0,3)--++(4,0)--++(0,-3)--++(1,0)--++(0,-2)--++(1,0)--++(0,2)--++(2.75,0);

\foreach \i in {0,...,9}
{
\foreach \j in {0,...,9}
{
\filldraw[blue](\i+1/4,\j+1/4) circle (0.08);
}
}
\end{tikzpicture}
\hspace{0.2in}
\begin{tikzpicture}[line width=.25mm,scale=0.4] 
\draw[dashed, line width=.1mm] (0,0)--(10,0)--(10,10)--(0,10)--cycle;

\draw[blue] ( 1/4 , 1/4 )--( 5/4 , 1/4 );
\draw[blue] ( 1/4 , 1/4 )--( 0 , 1/4 );
\draw[blue] ( 37/4 , 1/4 )--( 10 , 1/4 );
\draw[blue] ( 1/4 , 5/4 )--( 0 , 5/4 );
\draw[blue] ( 37/4 , 5/4 )--( 10 , 5/4 );
\draw[blue] ( 1/4 , 9/4 )--( 5/4 , 9/4 );
\draw[blue] ( 1/4 , 13/4 )--( 1/4 , 17/4 );
\draw[blue] ( 1/4 , 13/4 )--( 5/4 , 13/4 );
\draw[blue] ( 1/4 , 13/4 )--( 0 , 13/4 );
\draw[blue] ( 37/4 , 13/4 )--( 10 , 13/4 );
\draw[blue] ( 1/4 , 17/4 )--( 5/4 , 17/4 );
\draw[blue] ( 1/4 , 17/4 )--( 0 , 17/4 );
\draw[blue] ( 37/4 , 17/4 )--( 10 , 17/4 );
\draw[blue] ( 1/4 , 21/4 )--( 1/4 , 25/4 );
\draw[blue] ( 1/4 , 21/4 )--( 5/4 , 21/4 );
\draw[blue] ( 1/4 , 25/4 )--( 5/4 , 25/4 );
\draw[blue] ( 1/4 , 25/4 )--( 0 , 25/4 );
\draw[blue] ( 37/4 , 25/4 )--( 10 , 25/4 );
\draw[blue] ( 1/4 , 29/4 )--( 1/4 , 33/4 );
\draw[blue] ( 1/4 , 29/4 )--( 5/4 , 29/4 );
\draw[blue] ( 1/4 , 37/4 )--( 5/4 , 37/4 );
\draw[blue] ( 1/4 , 37/4 )--( 0 , 37/4 );
\draw[blue] ( 37/4 , 37/4 )--( 10 , 37/4 );
\draw[blue] ( 5/4 , 1/4 )--( 5/4 , 5/4 );
\draw[blue] ( 5/4 , 1/4 )--( 5/4 , 0 );
\draw[blue] ( 5/4 , 37/4 )--( 5/4 , 10 );
\draw[blue] ( 5/4 , 1/4 )--( 9/4 , 1/4 );
\draw[blue] ( 5/4 , 5/4 )--( 9/4 , 5/4 );
\draw[blue] ( 5/4 , 9/4 )--( 5/4 , 13/4 );
\draw[blue] ( 5/4 , 9/4 )--( 9/4 , 9/4 );
\draw[blue] ( 5/4 , 13/4 )--( 9/4 , 13/4 );
\draw[blue] ( 5/4 , 17/4 )--( 5/4 , 21/4 );
\draw[blue] ( 5/4 , 21/4 )--( 9/4 , 21/4 );
\draw[blue] ( 5/4 , 25/4 )--( 5/4 , 29/4 );
\draw[blue] ( 5/4 , 25/4 )--( 9/4 , 25/4 );
\draw[blue] ( 5/4 , 29/4 )--( 5/4 , 33/4 );
\draw[blue] ( 5/4 , 33/4 )--( 5/4 , 37/4 );
\draw[blue] ( 5/4 , 33/4 )--( 9/4 , 33/4 );
\draw[blue] ( 5/4 , 37/4 )--( 9/4 , 37/4 );
\draw[blue] ( 9/4 , 1/4 )--( 9/4 , 5/4 );
\draw[blue] ( 9/4 , 1/4 )--( 13/4 , 1/4 );
\draw[blue] ( 9/4 , 5/4 )--( 9/4 , 9/4 );
\draw[blue] ( 9/4 , 9/4 )--( 9/4 , 13/4 );
\draw[blue] ( 9/4 , 9/4 )--( 13/4 , 9/4 );
\draw[blue] ( 9/4 , 13/4 )--( 13/4 , 13/4 );
\draw[blue] ( 9/4 , 21/4 )--( 9/4 , 25/4 );
\draw[blue] ( 9/4 , 29/4 )--( 9/4 , 33/4 );
\draw[blue] ( 9/4 , 29/4 )--( 13/4 , 29/4 );
\draw[blue] ( 9/4 , 33/4 )--( 13/4 , 33/4 );
\draw[blue] ( 9/4 , 37/4 )--( 13/4 , 37/4 );
\draw[blue] ( 13/4 , 1/4 )--( 13/4 , 0 );
\draw[blue] ( 13/4 , 37/4 )--( 13/4 , 10 );
\draw[blue] ( 13/4 , 1/4 )--( 17/4 , 1/4 );
\draw[blue] ( 13/4 , 5/4 )--( 17/4 , 5/4 );
\draw[blue] ( 13/4 , 9/4 )--( 17/4 , 9/4 );
\draw[blue] ( 13/4 , 13/4 )--( 13/4 , 17/4 );
\draw[blue] ( 13/4 , 13/4 )--( 17/4 , 13/4 );
\draw[blue] ( 13/4 , 17/4 )--( 13/4 , 21/4 );
\draw[blue] ( 13/4 , 17/4 )--( 17/4 , 17/4 );
\draw[blue] ( 13/4 , 21/4 )--( 13/4 , 25/4 );
\draw[blue] ( 13/4 , 37/4 )--( 17/4 , 37/4 );
\draw[blue] ( 17/4 , 1/4 )--( 17/4 , 5/4 );
\draw[blue] ( 17/4 , 1/4 )--( 17/4 , 0 );
\draw[blue] ( 17/4 , 37/4 )--( 17/4 , 10 );
\draw[blue] ( 17/4 , 9/4 )--( 17/4 , 13/4 );
\draw[blue] ( 17/4 , 17/4 )--( 17/4 , 21/4 );
\draw[blue] ( 17/4 , 17/4 )--( 21/4 , 17/4 );
\draw[blue] ( 17/4 , 29/4 )--( 17/4 , 33/4 );
\draw[blue] ( 17/4 , 33/4 )--( 21/4 , 33/4 );
\draw[blue] ( 17/4 , 37/4 )--( 21/4 , 37/4 );
\draw[blue] ( 21/4 , 1/4 )--( 21/4 , 0 );
\draw[blue] ( 21/4 , 37/4 )--( 21/4 , 10 );
\draw[blue] ( 21/4 , 5/4 )--( 21/4 , 9/4 );
\draw[blue] ( 21/4 , 9/4 )--( 25/4 , 9/4 );
\draw[blue] ( 21/4 , 13/4 )--( 21/4 , 17/4 );
\draw[blue] ( 21/4 , 17/4 )--( 21/4 , 21/4 );
\draw[blue] ( 21/4 , 25/4 )--( 21/4 , 29/4 );
\draw[blue] ( 21/4 , 25/4 )--( 25/4 , 25/4 );
\draw[blue] ( 21/4 , 29/4 )--( 21/4 , 33/4 );
\draw[blue] ( 21/4 , 29/4 )--( 25/4 , 29/4 );
\draw[blue] ( 21/4 , 33/4 )--( 21/4 , 37/4 );
\draw[blue] ( 21/4 , 33/4 )--( 25/4 , 33/4 );
\draw[blue] ( 25/4 , 1/4 )--( 29/4 , 1/4 );
\draw[blue] ( 25/4 , 5/4 )--( 29/4 , 5/4 );
\draw[blue] ( 25/4 , 9/4 )--( 25/4 , 13/4 );
\draw[blue] ( 25/4 , 13/4 )--( 25/4 , 17/4 );
\draw[blue] ( 25/4 , 13/4 )--( 29/4 , 13/4 );
\draw[blue] ( 25/4 , 17/4 )--( 25/4 , 21/4 );
\draw[blue] ( 25/4 , 17/4 )--( 29/4 , 17/4 );
\draw[blue] ( 25/4 , 21/4 )--( 25/4 , 25/4 );
\draw[blue] ( 25/4 , 33/4 )--( 25/4 , 37/4 );
\draw[blue] ( 29/4 , 1/4 )--( 29/4 , 5/4 );
\draw[blue] ( 29/4 , 1/4 )--( 29/4 , 0 );
\draw[blue] ( 29/4 , 37/4 )--( 29/4 , 10 );
\draw[blue] ( 29/4 , 1/4 )--( 33/4 , 1/4 );
\draw[blue] ( 29/4 , 5/4 )--( 33/4 , 5/4 );
\draw[blue] ( 29/4 , 9/4 )--( 29/4 , 13/4 );
\draw[blue] ( 29/4 , 13/4 )--( 33/4 , 13/4 );
\draw[blue] ( 29/4 , 17/4 )--( 29/4 , 21/4 );
\draw[blue] ( 29/4 , 21/4 )--( 29/4 , 25/4 );
\draw[blue] ( 29/4 , 25/4 )--( 33/4 , 25/4 );
\draw[blue] ( 29/4 , 29/4 )--( 29/4 , 33/4 );
\draw[blue] ( 29/4 , 29/4 )--( 33/4 , 29/4 );
\draw[blue] ( 29/4 , 33/4 )--( 33/4 , 33/4 );
\draw[blue] ( 29/4 , 37/4 )--( 33/4 , 37/4 );
\draw[blue] ( 33/4 , 5/4 )--( 37/4 , 5/4 );
\draw[blue] ( 33/4 , 9/4 )--( 37/4 , 9/4 );
\draw[blue] ( 33/4 , 17/4 )--( 33/4 , 21/4 );
\draw[blue] ( 33/4 , 17/4 )--( 37/4 , 17/4 );
\draw[blue] ( 33/4 , 21/4 )--( 33/4 , 25/4 );
\draw[blue] ( 33/4 , 25/4 )--( 33/4 , 29/4 );
\draw[blue] ( 33/4 , 25/4 )--( 37/4 , 25/4 );
\draw[blue] ( 33/4 , 29/4 )--( 33/4 , 33/4 );
\draw[blue] ( 33/4 , 33/4 )--( 37/4 , 33/4 );
\draw[blue] ( 37/4 , 1/4 )--( 37/4 , 5/4 );
\draw[blue] ( 37/4 , 1/4 )--( 37/4 , 0 );
\draw[blue] ( 37/4 , 37/4 )--( 37/4 , 10 );
\draw[blue] ( 37/4 , 5/4 )--( 37/4 , 9/4 );
\draw[blue] ( 37/4 , 9/4 )--( 37/4 , 13/4 );
\draw[blue] ( 37/4 , 17/4 )--( 37/4 , 21/4 );
\draw[blue] ( 37/4 , 29/4 )--( 37/4 , 33/4 );

\draw[blue,line width=.75mm] (1.25,0)--++(0,1.25)--++(1,0)--++(0,2)--++(-2,0)--++(0,1)--++(1,0)--++(0,1)--++(1,0)--++(0,1)--++(-1,0)--++(0,3.75);

\foreach \i in {0,...,9}
{
\foreach \j in {0,...,9}
{
\filldraw[blue](\i+1/4,\j+1/4) circle (0.08);
}
}

\end{tikzpicture}
\caption{\label{fig:giant1}Two giant cycles for a random system of $1$-dimensional plaquettes (bonds) on a $2$-dimensional torus, shown in a square with opposite sides identified.}

    \includegraphics[width=.7\textwidth]{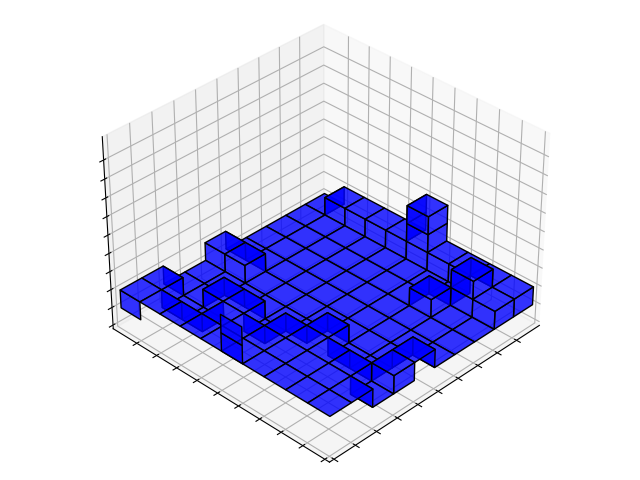}
    \caption{\label{fig:sheet}A giant cycle for $2$-dimensional plaquette percolation on a $3$-dimensional torus, shown in a cube with opposite sides identified.}
\end{figure}

For a random shape $S$ in the torus $\T^d$, let $\phi: S \hookrightarrow \T^d$ denote the natural inclusion map. Bobrowski and Skraba suggest that nontrivial elements in the image of the induced map $\phi_*:H_i\paren{S;\mathbb{Q}}\rightarrow H_i\paren{\T^d;\mathbb{Q}}$ can be considered ``giant $i$-dimensional cycles''. In the case  $i=1,$ these correspond to paths that loop around the torus as in Figure \ref{fig:giant1} whereas for the case $i=2$ and $d=3$ they are ``sheets'' going from one side of the torus to the other, as illustrated in Figure~\ref{fig:sheet}.

In~\cite{bobrowski2020homological}, Bobrowski and Skraba provide experimental evidence that the appearance of giant cycles is closely correlated with the zeroes of the expected Euler characteristic curve, and moreover that this behavior seems universal across several models. In~\cite{bobrowski2020homological2} they focus on continuum percolation and show that giant cycles appear in all dimensions within the so-called thermodynamic limit where $nr^d$ is bounded. One interesting suggestion is that there is a sharp, convergent transition from the map $\phi_*$ being trivial to being surjective. They prove that such a transition occurs when $i=1$, for every $d \ge 2$.

The Harris--Kesten theorem~\cite{harris1960lower,Kesten80} (see also~\cite{bollobas2006short} for a short proof) establishes that for bond percolation in the square lattice, the critical probability for an infinite component to appear is $p=1/2$. Kesten also showed that $p_c = 1/ 2$ for site percolation in the triangular lattice, using similar tools~\cite{Kesten-book}.

We seek higher-dimensional homological analogues of these classical theorems. We take a step in this direction by proving finitary versions on certain flat tori with a high degree of symmetry. We will discuss first an analogue of bond percolation on a square lattice, and then an analogue of site percolation on a triangular lattice.

A natural analogue of bond percolation in higher dimensions is given by \emph{plaquette percolation}, which was first studied by Aizenman, Chayes, Chayes, Fr\"{o}hlich, and Russo in~\cite{ACCFR83}. They were motivated by notions of random surfaces in physics coming from lattice gauge theories and three-dimensional spin systems---see Sections 1 and 7 of that paper. The plaquette model starts with the entire 1-skeleton of $\Z^3$ and adds 2-dimensional square cells, or plaquettes, with probability $p$ independently. The authors prove that the probability that a large planar loop is null-homologous undergoes a phase transition from an ``area law'' to a ``perimeter law'' that is dual to the phase transition for bond percolation in $\Z^3.$ In particular, the critical probability for this threshold is at $p_c = 1-\hat{p}_c$, where $\hat{p}_c$ is the threshold for bond percolation (this follows when their results are combined with those of~\cite{grimmett1990supercritical}). They also show that infinite ``surface sheets'' appear at this threshold. At the end of their paper, the authors defined $i$-dimensional plaquette percolation on $\Z^d$ and suggested the study of analogous questions in higher dimensions.
\begin{quote}
Of particular interest are random $(d/2)$-cells in even dimension $d.$ Clearly, if no
intermediate phase exists in such a self-dual system, the transition point is $p = 1/2.$
The most promising model for future study is random plaquettes in $d = 4.$~\cite{ACCFR83}
\end{quote}

One of the models we study is the following \emph{$i$-dimensional plaquette percolation model} $P$ on the $d$-dimensional torus. Let $N \ge 1$ and consider the cubical complex $T^d_N = \Z^d / (N\Z)^d.$ Take the entire $(i-1)$-skeleton of $T^d_N$ and add each $i$-face independently with probability $p.$

We also study an analogue of site percolation in a triangular lattice. In general, site percolation is defined as a random induced subgraph of a graph where each vertex is open with probability $p$, independently. For the special case of site percolation on a triangular lattice, it is often convenient to consider a ``dual'' model where one tiles the plane by regular hexagons, and every closed hexagon is included independently with probability $p$. As observed by Bobrowski and Skraba in~\cite{bobrowski2020homological}, this dual model generalizes naturally to higher dimensions---indeed, one can tile Euclidean space $\R^d$ by regular permutohedra, and then include each closed permutohedron independently with probability $p.$ More details, including a topological justification for referring to this as site percolation, are given below.

\subsection{Main results}

Our first set of results is in the case of plaquette percolation on the torus. Let $d,N\in\N,$ $d>1,$ $1\leq i \leq d-1$ and $p\in\paren{0,1}.$ As above, denote by  $T^d_N = \Z^d / (N\Z)^d$ the regular cubical complex on the $d$-dimensional torus with $N^d$ cubes of width one. Let $H_i\paren{X}$ be homology with coefficients in a fixed field $\F.$ A brief review of homology can be found in Section~\ref{sec:homology}. The choice of $\F$ only matters in the arguments of Section~\ref{sec:surjective}, which is reflected in the minor restrictions on the characteristic of $\F$ (denoted $\mathrm{char}\paren{\F}$) stated in the main theorems. Roughly speaking, our proofs rely on three symmetries: the symmetry of the plaquette system with its dual, the symmetry of individual plaquettes (that is, the existence of a transitive group action on the set of plaquettes), and the symmetry properties of the homology groups of the torus. The hypotheses on the homology coefficient field are crucial for the third symmetry; they ensure that the homology groups of the torus are maximally symmetric in the sense that they are irreducible representations of the point symmetry group of the lattice. This ensures that there is a sharp transition from the appearance of the first non-zero giant cycle to the appearance of a basis of giant cycles. See the third paragraph of Section~\ref{sec:examples} for an example.

Define the plaquette system $P=P\paren{i,d,N,p}$ to be the random set obtained by taking the $(i-1)$-skeleton of $\T^d_N$ and adding each $i$-face independently with probability $p.$ Let $\phi:P \hookrightarrow \T^d_N$ be the inclusion, and let $\phi_*:H_{i}\paren{P}\rightarrow H_i\paren{\T^d_N}$ be the induced map on homology (defined in Section~\ref{sec:homology}). Also, denote by $A^{\square}=A^{\square}(i,d,N,p)$ the event that $\phi_*$ is nontrivial, and denote by $S^{\square}=S^{\square}(i,d,N,p)$ the event that $\phi_*$ is surjective (we will hexagon superscripts for the corresponding events for permutohedral site percolation). For example, in Figure \ref{fig:giant1} the two giant cycles shown are homologous with standard generators for $H_1(\T^2)$, so we have the event $S^{\square}$.

Our main result for plaquette percolation is that if $d=2i,$ then $i$-dimensional percolation is self-dual and undergoes a sharp transition at $p=1/2.$

\begin{Theorem}
\label{thm:half}
Suppose $\mathrm{char}\paren{\F} \neq 2.$ If $d=2i$ then 
\[\begin{cases}
\mathbb{P}_p\paren{A^{\square}}\rightarrow 0 & p<\frac{1}{2} \\
\mathbb{P}_p\paren{S^{\square}}\rightarrow 1 & p>\frac{1}{2} \\
\end{cases}\]
as $N\rightarrow \infty.$
\end{Theorem}

Using results on bond percolation on $\Z^d$, we also prove dual sharp thresholds for $1$-dimensional and $(d-1)$-dimensional percolation on the torus.

\begin{Theorem}
\label{thm:one}
Let $\hat{p}_c=\hat{p}_c(d)$ be the critical threshold for bond percolation on $\Z^d.$ If $i=1$ then
\[\begin{cases}
\mathbb{P}_p\paren{A^{\square}}\rightarrow 0 & p<\hat{p}_c \\
\mathbb{P}_p\paren{S^{\square}}\rightarrow 1 & p>\hat{p}_c \\
\end{cases}\]
as $N\rightarrow \infty.$

Furthermore, if $i=d-1$ then 
\[\begin{cases}
\mathbb{P}_p\paren{A^{\square}}\rightarrow 0 & p<1-\hat{p}_c \\
\mathbb{P}_p\paren{S^{\square}}\rightarrow 1 & p>1-\hat{p}_c \\
\end{cases}\]
as $N\rightarrow \infty.$
\end{Theorem}
In the above, we also show that the decay of $\mathbb{P}_p\paren{A^{\square}}$ below the threshold and  $\mathbb{P}_p\paren{S^{\square}}$ above the threshold is exponentially fast for both $i=1$ and $i=d-1.$

For other values of $i$ and $d$ we show the existence of a sharp threshold function as follows. For each $N \in \N,$ let $\lambda^{\square}\paren{N,i,d}$ satisfy
\begin{equation}
\label{eqn:lambda}
\mathbb{P}_{\lambda^{\square}\paren{N,i,d}}\paren{A^{\square}} = \frac{1}{2}\,.
\end{equation}
Note that $\mathbb{P}_p\paren{A^{\square}}$ is continuous as a function of $p,$ so such a $\lambda^{\square}\paren{N,i,d}$ exists by the intermediate value theorem. Since the tori of different sizes do not embed nicely into each other, it is not obvious that $\lambda^{\square}\paren{N,i,d}$ should be convergent a priori. 

We should mention that this choice of $\lambda^{\square}\paren{N,i,d}$ is somewhat arbitrary. We could replace $\frac{1}{2}$ in Equation~\ref{eqn:lambda} with any constant strictly between $0$ and $1,$ for example, and the sharp threshold results we use would apply just as well. In several cases we show that $\lambda^{\square}\paren{N,i,d}$ converges, and in those cases the limiting value could also be taken as a constant threshold function.

Define
\[p^{\square}_c\paren{i,d}=\limsup_{N\rightarrow\infty} \lambda^{\square}\paren{N}\]
and 
\[q^{\square}_c\paren{i,d}=\liminf_{N\rightarrow\infty} \lambda^{\square}\paren{N}\,.\]

With the understanding that these depend on choice of $i$ and $d$, which are always understood in context, we sometimes abbreviate to simply $p^{\square}_c,$ $q^{\square}_c,$ and $\lambda^{\square}(N).$
\begin{Theorem}
\label{thm:weak}

Suppose $\mathrm{char}\paren{\F} \neq 2.$ For every $d \ge 2,$ $1 \le i \le d-1,$ and $\epsilon > 0$
\[\begin{cases}
\mathbb{P}_{\lambda^{\square}\paren{N}-\epsilon}\paren{A^{\square}}\rightarrow 0 \\
\mathbb{P}_{\lambda^{\square}\paren{N}+\epsilon}\paren{S^{\square}}\rightarrow 1 \\
\end{cases}\]
as $N\rightarrow \infty.$

Moreover, for every $d \ge 2$ and $1 \le i \le d-1$ we have 
\[0 < q^{\square}_c \leq p^{\square}_c < 1\]
and $p^{\square}_c\paren{i,d}$ has the following properties.
\begin{enumerate}[label=(\alph*)]
    \item (Duality) $p^{\square}_c\paren{i,d}+q^{\square}_c\paren{d-i,d} =1.$
    \item (Monotonicity in $i$ and $d$) $p^{\square}_c\paren{i,d} < p^{\square}_c\paren{i,d-1}< p^{\square}_c\paren{i+1,d}$ if\\ $0<i<d-1.$ 
\end{enumerate}
\end{Theorem}

It follows that $p^{\square}_c=q^{\square}_c$ for $i=d/2$, $i=1$, and $i=d-1$, and we conjecture that this equality (and hence sharp threshold from a trivial map to a surjective one at a constant value of $p$) holds for all $i$ and $d.$ Bobrowksi and Skraba make analogous conjectures for the continuum percolation model in~\cite{bobrowski2020homological2}. 

We also apply our methods to Bernoulli site percolation on the tiling of the torus by $d$-dimensional permutohedra, which was previously studied in~\cite{bobrowski2020homological}. The precise definitions are as follows. Let $$\hat{\R}^d \coloneqq \set{(x_0,x_1,\ldots,x_d) : \sum_{k=0}^d x_k = 0}.$$ Recall that the root lattice $\mathcal{A}_d$ is defined by
$$\mathcal{A}_d \coloneqq \hat{\R}^d \cap \Z^{d+1}.$$ The dual lattice is then defined by
$$\mathcal{A}_d^* \coloneqq \set{x \in \hat{\R}^d : \forall y \in \mathcal{A}_d, x\cdot y \in \Z}$$ which is generated by the basis
\begin{align*}
    B \coloneqq \set{\vec{1}-d\vec{e}_k: 1\leq k \leq d}\,.
\end{align*}
Here $\vec{1}$ is the vector whose entries all equal $1$. 
The closed Voronoi cells of $\mathcal{A}^*_d$ in $\hat{\R}^d$ are $d$-dimensional permutohedra. When $d=2,$ $\mathcal{A}_2^*$ is the triangular lattice and the permutohedra are hexagons. For the case $d=3,$  $\mathcal{A}_3^*$ is the body-centered cubic lattice and the permutohedra are truncated octahedra (see~\cite{baek2009some} for a detailed exposition).

Consider the torus ${\bf T}^d_N$ as the parallelepiped generated by $\set{Nv : v \in B}$ with opposite faces identified (we use boldface to distinguish from the plaquette structure on the torus). Define $Q = Q(d,N,p)$ to be the random set obtained by adding each permutohedron independently with probability $p.$ The topological justification for calling this site percolation is that the adjacency graph on the permutohedra of $Q$ is site percolation on the lattice $\mathcal{A}_d^*,$ a generalization of how site percolation on the triangular lattice is sometimes studied instead as face percolation on hexagons. Moreover, the topology of $Q$ can be recovered from the information encoded in its adjacency graph as described in the second paragraph of Section~\ref{sec:monotonicity}. This key property is not shared by site percolation on the cubical lattice. 

The giant cycle events are defined as before, except that $i$-dimensional giant cycles exhibit interesting behavior for all $1\leq i\leq d-1$ (in the plaquette model $P\paren{i,d,N,p},$ all giant cycles in homological dimensions less than $i$ are automatically present, and there can be no giant cycles in homological dimensions exceeding $i.$) More precisely, let $\varphi : Q \hookrightarrow {\bf T}^d_N$ be the inclusion, and let $\varphi_{i*} : H_i\paren{Q} \to H_i\paren{{\bf T}^d_N}$ be the induced maps for homology in each dimension. For each $i,$ Let $A^{\hexagon}_i$ be the event that $\varphi_{i*}$ is nonzero and let $S^{\hexagon}_i$ be the event that $\varphi_{i*}$ is surjective. 

For each $N \in \N,$ let $\lambda_i^{\hexagon}\paren{N,d}$ satisfy
\[\mathbb{P}_{\lambda_i^{\hexagon}\paren{N,d}}\paren{A_i^{\hexagon}} = \frac{1}{2}\,.\]
Define
\[p^{\hexagon}_i\paren{d}=\limsup_{N\rightarrow\infty} \lambda_i^{\hexagon}\paren{N,d}\]
and 
\[q^{\hexagon}_i\paren{d}=\liminf_{N\rightarrow\infty} \lambda_i^{\hexagon}\paren{N,d}\,.\]

\begin{Theorem}\label{thm:permutohedral}
Suppose $\mathrm{char}\paren{\F}=0$ or that $\mathrm{char}\paren{\F}$ is an odd prime not dividing $d+1.$ For every $d \ge 2,$ $1 \le i \le d-1,$ and $\epsilon > 0$
    \[\begin{cases}
    \mathbb{P}_{\lambda_i^{\hexagon}\paren{N}-\epsilon}\paren{A_i^{\hexagon}}\rightarrow 0 \\
    \mathbb{P}_{\lambda_i^{\hexagon}\paren{N}+\epsilon}\paren{S_i^{\hexagon}}\rightarrow 1 \\
    \end{cases}\]
    as $N\rightarrow \infty.$
    For every $d \ge 2$ and $1 \le i \le d-1$ we have 
    \[0<q_i^{\hexagon} \leq  p_i^{\hexagon}<1\,.\]
    In some cases $\lambda_i^{\hexagon}$ converges, namely 
    \[\lim_{N\to\infty} \lambda_1^{\hexagon}\paren{N} = p_c(\mathcal{A}_d^*)\,,\]
    \[\lim_{N\to\infty} \lambda_{d-1}^{\hexagon}\paren{N} =1-p_c(\mathcal{A}_d^*)\,,\]
    and if $d$ is even, then \[\lim_{N\to\infty} \lambda_{d/2}^{\hexagon}\paren{N} = \frac{1}{2}\,,\]
    even if we drop the assumption on $\mathrm{char}\paren{\F}.$
    Moreover, $p^{\hexagon}_i\paren{d}$ has the following properties.
\begin{enumerate}[label=(\alph*)]
    \item (Duality) $p^{\hexagon}_i\paren{d}+q^{\hexagon}_{d-i}\paren{d} =1.$
    \item (Monotonicity in $i$ and $d$) $p^{\hexagon}_i\paren{d} < p^{\hexagon}_i\paren{d-1}< p^{\hexagon}_{i+1}\paren{d}$ if\\ $0<i<d-1.$ 
\end{enumerate}
\end{Theorem}

In particular, when $d=4,$ the random set $Q$ exhibits three qualitatively distinct phase transitions at  $p_c(A_4^*), \frac{1}{2},$ and  $1-p_c(A_4^*),$ where $p_c(A_4^*)$ is the site percolation threshold for the lattice $A_4^*.$

\subsection{Review of Homology}\label{sec:homology}

We provide a brief review of the definition of homology with coefficients in a field $\F.$ We focus on the case of subcomplexes of $\T^d_N,$ which are relevant to plaquette percolation, though the definition easily generalizes to other regular CW complexes such as those obtained by site percolation on the permutohedral lattice.
For a more comprehensive introduction to homology, please see a reference such as~\cite{hatcher2002algebraic}. A hands-on treatment specific to cubical complexes can also be found in Chapter 2 of~\cite{kaczynski2004computational}.

A \emph{subcomplex} of $\T^d_N$ is a set $P$ that is the union of plaquettes in $\T^d_N$ of various dimensions, with the property that if $\tau$ is in the boundary of $\sigma$ and $\sigma\subset P$ then $\tau\subset P.$ Denote the set of $i$-plaquettes of $P$ by $P^{\paren{i}}.$ Fix a field $\F$. For $0\leq i \leq d$ define the $i$-dimensional chain group of $P$ with coefficients in $\F$ as the vector space of formal $\F$-linear combinations of the $i$-plaquettes of $P$:
\[C_i\paren{P}=C_i\paren{P;\F}\coloneqq \set{\sum{\sigma}\in P^{(i)}c_{\sigma}\sigma : c_{\sigma}\in \F}\,.\]
The \emph{boundary maps} $\partial_i:C_i\paren{P}\rightarrow C_{i-1}\paren{P}$ relate chain groups of different dimensions. Roughly speaking, if $\sigma$ is an $i$-plaquette then $\partial_i\sigma$ is the oriented sum of $(i-1)$-plaquettes adjacent to $\sigma.$ The signs corresponding to the orientation must be chosen carefully in order so that the boundary operator satisfies the fundamental relation

\begin{equation}
\label{eq:bd}
\partial_{i}\circ \partial_{i-1}=0\,.
\end{equation}

While the boundary of $0$-plaquette $v$ is empty, the boundary of a $1$-plaquette $\paren{v_0,v_1}$ is $v_1-v_1,$ and the boundary of a two plaquette $\paren{v_0,v_1,v_2,v_3}$ is $\paren{v_0,v_1}+\paren{v_1,v_2}+\paren{v_2,v_3}-\paren{v_0,v_3},$ the general formula is more complicated. We must first define a notation for $i$-plaquettes. Let $J=J\paren{\sigma}=\set{j_1,\ldots,j_i}$ where $1\leq j_1 < j_2 \ldots < j_i \leq d.$  For $k\in J,$ set $I_k=\brac{x_k,x_k+1}$ for some $x_k\in \set{0,\ldots,N-1}.$ On the other hand, for $k\in \set{1,\ldots,d}\setminus J,$ let $I_k=\set{x_k}$ for some $x_k\in \set{0,\ldots,N-1}.$ Then a general $i$-plaquette in $\T_d^N$ can be written uniquely in the form $\sigma = \prod_{k=1}^{d} I_k.$ We can now define the boundary $\partial_i \sigma$ of a plaquette $\sigma \subset [0,1]^d$ by
\[\sum_{l=1}^i \paren{-1}^{l-1} \brac{\prod_{1 \leq k < j_l} I_k \times \set{1} \times \prod_{j_l < m \leq d} I_m - \prod_{1 \leq k < j_l} I_k \times \set{0} \times \prod_{j_l < m \leq d} I_m}\,,\]
and by translating appropriately and extending linearly we define the boundary map $\partial_i:C_i\paren{P}\rightarrow C_{i-1}\paren{P}.$ 

\begin{figure}
\centering
\includegraphics[width=.4\textwidth]{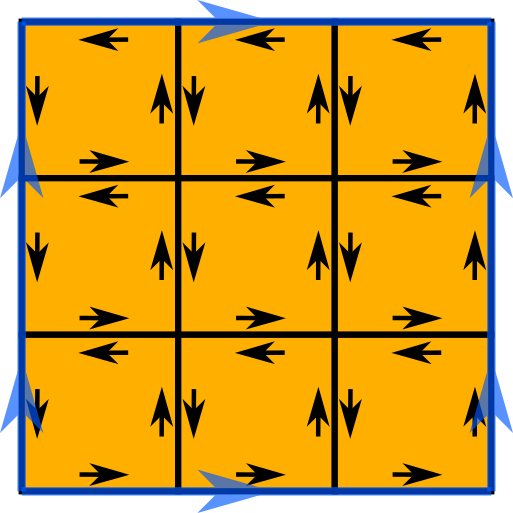}
\caption{The torus $\T^2_3$ consisting of nine $2$-plaquettes. The opposite sides of the larger square are identified so that the pair of single blue arrows match and pair of double blue arrows are each matched. The formal sum of the oriented edges indicated by the four black arrows on a $2$-plaquette is its image under the boundary map $\partial_2.$ Note that the sum of the boundaries of all of the $2$-plaquettes is zero, so the sum of all of the plaquettes is an element of $Z_2\paren{\T^2_3}.$}
\label{fig:subdivision}
\end{figure}

Note that the dimension of $C_i\paren{\T^d_N}$ is the number of $i$-plaquettes of $\T^d_N,$ which grows as a polynomial in $N.$ However, we can use the boundary maps to define a quantity which does not depend on $N,$ and will be topologically invariant more generally. The space of \emph{$i$-dimensional cycles} of $P$ is the kernel of the linear transformation $\partial_i$: $Z_i\paren{P}=\mathrm{Ker}\,\partial_i$ and the space of \emph{$i$-dimensional boundaries} is the image of $\partial_{i+1}: B_{i}\paren{P}=\mathrm{Im}\,\partial_{i+1}.$ By Equation~\ref{eq:bd}, we have that $B_i\paren{P}\subseteq Z_i\paren{P}$ and the quotient vector space $Z_i\paren{P}/B_i\paren{P}$ is well-defined. This quotient of cycles modulo boundaries is the \emph{$i$-dimensional homology} of $P$ and is denoted by $H_i\paren{P;\F}$ or $H_i\paren{P}$ when the coefficient field is given by context. 

For example, consider $\T^2_3$ shown in Figure~\ref{fig:subdivision}. One can check that the only chains $a\in C_{2}\paren{\T^2_3}$ that satisfy $\partial_2 a=0$ are those that give the same coefficient to each $2$-plaquette of $\T^2_3.$ Thus $Z_2\paren{\T^2_3}$ is one-dimensional. Moreover, $C_3\paren{\T^2_3}=0,$ it follows that $B_2\paren{\T^2_3}=0$ and thus $H_2\paren{\T^2_3}$ is also one-dimensional. A more laborious computation reveals that $H_1\paren{\T^2_3}$ is two-dimensional and is generated by oriented paths at on the left side and bottom of Figure~\ref{fig:subdivision}. More generally, it turns out that $H_i\paren{\T^d}$ is a vector space of dimension $\binom{d}{i},$ and is generated by $i$-planes in the coordinate directions of $\T^d_N.$ In theory, this can be found by performing an even more laborious computation using the cubical complex structure on $\T^d_3.$  However, in practice, the homology of $\T^d$ is computed using the K\"{u}nneth formula for homology, which gives an expression for the homology of a product space in terms of the homology of its factors (see Section 3b of~\cite{hatcher2002algebraic}). We describe $H_i\paren{\T^d}$ at the beginning of Section~\ref{sec:examples}.

We can now give a formal definition of a giant cycle. Let $P$ be a subcomplex of $\T^d_N$ and let $a\in Z_i\paren{P}$ be an $i$-cycle representing an equivalence class $[a]\in H_i\paren{P}.$ Then $a$ is also an $i$-cycle of $\T^d_N.$ Moreover, the equivalence class of of $a$ in $H_i\paren{\T^d}=Z_i\paren{\T^d}/B_i\paren{\T^d}$ only depends on the equivalence class of $a$ in $H_i\paren{P}=Z_i\paren{P}/B_i\paren{P}$ because $B_i\paren{P}\subseteq B_i\paren{\T^d_N}.$ More formally, if $a,b\in Z_i\paren{P}$ are in the same homology class in $H_i\paren{P}$ then $b=a+\partial_{i+1} c$ for some $c\in C_{i+1}\paren{P}.$ But $c$ is also an element of $C_{i+1}\paren{\T^d_N},$ so $a$ and $b$ are in the same homology class in $H_i\paren{\T^d_N}.$ Therefore, we have a well-defined map on homology $\phi_*:H_i\paren{P}\rightarrow H_i\paren{\T^d_N},$ as illustrated in Figure~\ref{fig:giantcycledefinition}. Then $a$ is a \emph{giant cycle} if $\phi_*\paren{a}\neq 0$ or, in other words, if $a$ is not a boundary in $H_i\paren{\T^d_N}.$

\begin{figure}
\centering
\includegraphics[width=.8\textwidth]{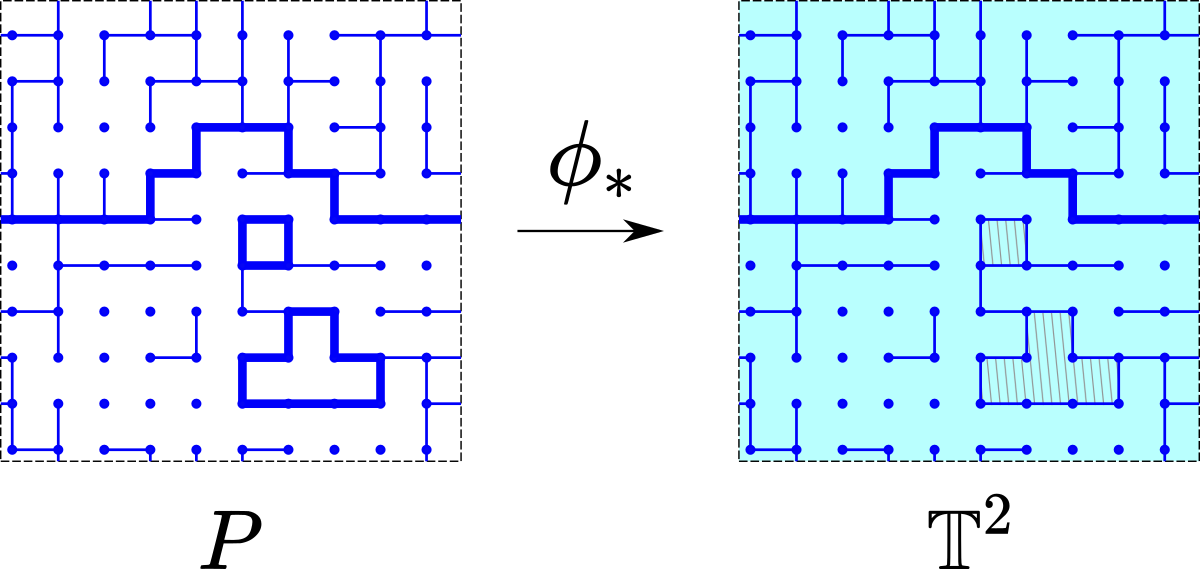}
\caption{The two torus $\T^2_{10}$ and a subcomplex $P.$ shown in a square with periodic boundary conditions. Three cycles of $Z_1\paren{P}$ are illustrated in bold on the right, each of which represents a different non-zero homology class. When these cycles are mapped to $H_1\paren{\T^2_{10}}$ by $\phi_*,$ two of them become boundaries and one (shown in bold) remains a non-zero cycle and is thus a giant cycle of $P.$  }
\label{fig:giantcycledefinition}
\end{figure}

\subsection{Probabilistic tools}
Let $X$ be a finite set and let $\mu_p$ be the $\mathrm{Bernoulli}(p)$ product measure on the power set $\mathcal{P}\paren{X}$ defined by including each element of $X$ independently with probability $p.$ That is, if $Y\subseteq X,$ 
\[\mu\paren{Y}=p^{\abs{Y}}\paren{1-p}^{\abs{X}-\abs{Y}}\,.\]
An event $B$ is \emph{increasing} if 
\[Y_0\subset Y_1, Y_0\in B\implies Y_1\in B\,.\]
We require Harris's Inequality on increasing events~\cite{harris1960lower}, which is a special case of the FKG Inequality~\cite{fortuin1971correlation}.    
\begin{Theorem}[Harris's Inequality]
If $B_1,\ldots, B_j$ are increasing events then
\[\mathbb{P}\paren{\bigcap_{k=1}^j B_k} \geq \prod_{k=1}^j \mathbb{P}\paren{B_k}\,.\]
\end{Theorem}

Another key tool for us is the following theorem of Friedgut and Kalai on sharpness of thresholds~\cite{friedgut1996every}. Recall that the action of a group $G$ on $X$ is transitive if for every $x,y \in X,$ there is a $g \in G$ such that $g \cdot x = y.$

\begin{Theorem}[Friedgut and Kalai]
\label{thm:FK}
There exists a universal constant $\rho>0$ so that the following holds: Let $B$ be an increasing event that is invariant under a transitive group action on $X.$ If $\mu_p\paren{B}>\epsilon>0$ and 
\begin{equation}
    \label{eqn:FK}
q\geq p+\rho \frac{\log\paren{1/\paren{2\epsilon}}}{\log\paren{\abs{X}}}
\end{equation}
then $\mu_q\paren{B}>1-\epsilon.$
\end{Theorem}

We use two more technical results on connection probabilities in the subcritical and supercritical phases in bond percolation in $\Z^d$ below.  For clarity, we state them in Section~\ref{sec:one} when they are needed.
    
\subsection{Definitions and notation}    
To define the \emph{dual system of plaquettes} $P^{\blacksquare}=P^{\blacksquare}\paren{i,d,N,p}$, let $(\T^d_N)^{\blacksquare}$ be the regular cubical complex obtained by shifting  $\T^d_N$ by $\frac{1}{2}$ in each coordinate direction. Each $i$-face of $\T^d_N$ intersects a unique $(d-i)$-face of $(\T^d_N)^{\blacksquare}$ and they meet in a single point at their centers. For example, the faces $\brac{0,1}^i\times\set{0}^{d-i}$ and $\set{1/2}^{i}\times\brac{-1/2,1/2}^{d-i}$ intersect in the point $\set{\frac{1}{2}}^i\times \set{0}^{d-i}.$ 

Define the dual system $P^{\blacksquare}$ to be the subcomplex of $(\T^d_N)^{\blacksquare}$ consisting of all faces for which the corresponding face in  $\T^d_N$ is not contained in $P$. See Figure~\ref{fig:critical}. Observe that the distribution of $P^{\blacksquare}\paren{i,d,N,p}$ is the same as that of $P\paren{d-i,d,N,1-p}.$ If $B^{\blacksquare}$ is an event defined in terms of $P^{\blacksquare}$ we will write $\mathbb{P}_p\paren{B^{\blacksquare}}$ to mean the probability of $B^{\blacksquare}$ with respect to the parameter $p$ of $P.$

\begin{figure}
\centering
\includegraphics[width=.5\textwidth]{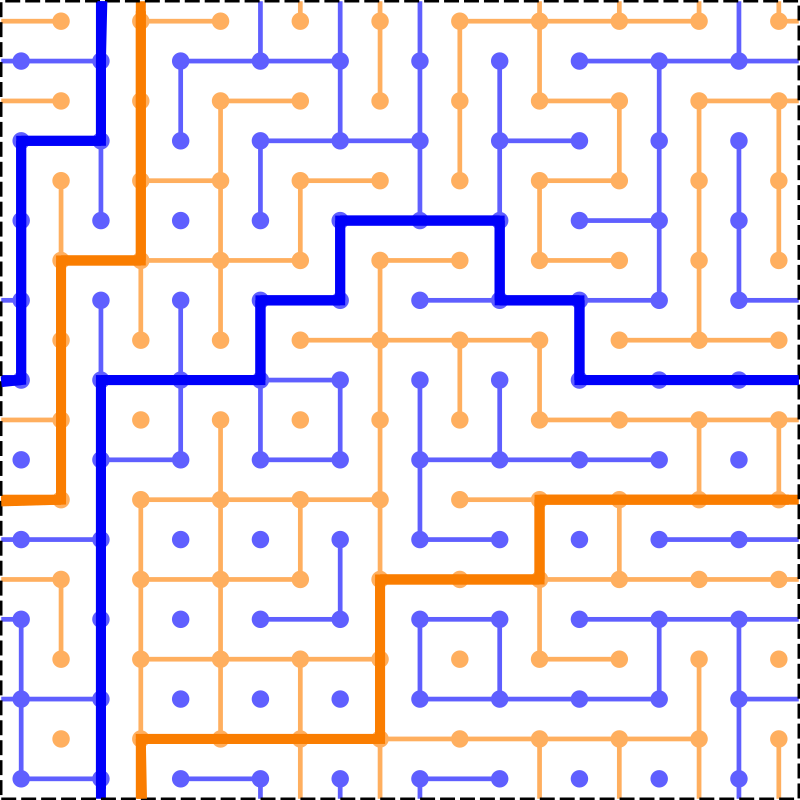}
\caption{Bond percolation at criticality (i.e.\ $p=1/2$) on the torus $T^2_{10}$ in blue, with the corresponding dual system of bonds in orange. Giant cycles are shown in bold. Observe that while $\rank\phi_*+\rank\psi_*=2$ (as required by duality), neither the bond system nor its dual has a giant cycle homologous to one of the standard basis elements of $H_1(\T^2).$}
\label{fig:critical}
\end{figure}

We always use the notation $\phi:P \hookrightarrow \T^d$ and $\psi:P^{\blacksquare} \hookrightarrow \T^d$ for the respective inclusion maps, and $\phi_*:H_{i}\paren{P}\rightarrow H_i\paren{\T^d_N}$ and $\psi_*:H_{d-i}\paren{P^{\blacksquare}}\rightarrow H_{d-i}\paren{\T^d}$ for the induced maps on homology. Also, we consistently use notation $A^{\square}=A^{\square}(i,d,N,p)$ for the event that $\operatorname{im}\phi_*\neq 0,$ $S^{\square}=S^{\square}(i,d,N,p)$ for the event that $\phi_*$ is surjective, and $Z^{\square}=Z^{\square}(i,d,N,p)$ for the event that $\phi_*$ is zero. Denote by $A^{\blacksquare}, S^{\blacksquare},$ and $Z^{\blacksquare}$ the corresponding events for $\psi_*.$

\subsection{Proof sketch}

We provide an overview of our main argument. Much of it is the same, mutatis mutandis, whether we are working with plaquettes or permutohedra. We will first prove our results with plaquettes, and then cover the changes that need to be made with permutohedra in Section~\ref{sec:permutohedra}. In the section on topological results (Section~\ref{sec:topological}), we show that duality holds in the sense that $\rank\phi_*+\rank\psi_*=\rank H_i\paren{\T^d}$  (Lemma~\ref{lemma:duality}). This is similar in spirit to results of~\cite{bobrowski2020homological} and~\cite{bobrowski2020homological2} for other models of percolation on the torus including permutohedral site percolation. In particular, at least one of the events $A^{\square}$ and $A^{\blacksquare}$ occurs, $S^{\square}$ occurs if and only if $Z^{\blacksquare}$ occurs, and  $S^{\blacksquare}$ occurs if and only if $Z^{\square}$ occurs. 

Our strategy is to exploit the duality between the events $S^{\square}$ and $Z^{\blacksquare}=\paren{A^{\blacksquare}}^c.$ Toward that end, we show that a threshold for $A^{\square}$ is also a threshold for $S^{\square}$ in Section~\ref{sec:surjective}. First, we use the action of the point symmetry group, i.e. the group of symmetries of $\T_d^N$ which preserve the origin, on the homology groups of the torus to show that there are constants $b_0$ and $b_1$ so that $\mathbb{P}_p\paren{S^{\square}}\geq b_0\mathbb{P}_p\paren{A^{\square}}^{b_1}.$ This follows from a more general result for events defined in terms of an irreducible representation of the point symmetry group (Lemma~\ref{lemma:spinning2}) and the fact that $H_i\paren{\T^d; \F}$ is an irreducible representation of the point symmetry group of $\Z^d$ assuming the characteristic of $\F$ does not equal $2$ (Proposition~\ref{prop:Zdirrep}). This is one point at which the argument differs for site percolation on the permutohedral lattice; to account for the symmetries of that lattice we include a different argument that assumes that the characteristic of $\F$ is not divisible by $d+1$ (Proposition~\ref{prop:permutoirrep}). 

Recall that $\lambda^{\square}$ was chosen such that
\[\mathbb{P}_{\lambda^{\square}\paren{N,i,d}}\paren{A^{\square}} = \frac{1}{2}\,.\]
By the above, it follows  that  $\mathbb{P}_{\lambda^{\square}\paren{N}}\paren{S^{\square}}>b_0 \paren{\frac{1}{2}}^{b_1}.$ $S^{\square}$ is increasing and invariant under the symmetry group of $\T^d$ so Friedgut and Kalai's theorem on sharpness of thresholds (Theorem~\ref{thm:FK}) implies that for any $\epsilon > 0,$ $\mathbb{P}_{\lambda^{\square}\paren{N}+\epsilon}\paren{S^{\square}}\rightarrow 1$ as $N\rightarrow\infty$.\looseness=-1

The proof of Theorem~\ref{thm:half} is then straightforward (Section~\ref{sec:half}).  By duality 
\[\mathbb{P}_{1/2}\paren{A^{\square}}=\mathbb{P}_{1/2}\paren{A^{\blacksquare}} \qquad \text{and} \qquad \mathbb{P}_{1/2}\paren{A^{\square}}+\mathbb{P}_{1/2}\paren{A^{\blacksquare}}\geq 1\,,\] so $\mathbb{P}_{1/2}\paren{A^{\square}}\geq \frac{1}{2}$ for all $N.$ It follows from the previous argument that $\mathbb{P}_p\paren{S^{\square}}\rightarrow 1$ for $p>1/2.$ On the other hand, if $p<1/2$ duality implies that 
\[\mathbb{P}_{p}\paren{A^{\square}} =1-\mathbb{P}_{p}\paren{S^{\blacksquare}}=1-\mathbb{P}_{1-p}\paren{S^{\square}}\rightarrow 0\,.\] 

Next, in Section~\ref{sec:duality} we study the relationship between duality and convergence. 
We show that $p^{\square}_c\paren{i,d}+q^{\square}_c\paren{d-i,d}=1$ by using Lemma~\ref{lemma:duality} and applying Theorem~\ref{thm:FK} to $A^{\square}$ above $p^{\square}_c\paren{d-i,d}$ and to $A^{\blacksquare}$ below $q^{\square}_c\paren{i,d}$ (Proposition~\ref{prop:duality}). It follows that the threshold for $A^{\square}$ converges if and only if $p^{\square}_c\paren{d-i,d}+p^{\square}_c\paren{i,d}=1$ (Corollary~\ref{cor:sharpness}).

In Section~\ref{sec:one} we show that homological percolation occurs at the critical threshold for bond percolation on $\Z^d$ by applying classical results on connection probabilities in the subcritical and supercritical regimes. This together with Lemma~\ref{lemma:duality} demonstrates Theorem~\ref{thm:one}. The proof does not use the results of Section~\ref{sec:surjective}, so we are able to drop the assumption on $\mathrm{char}\paren{\F}.$ 

Finally, in Section~\ref{sec:weak}, we complete the proof of Theorem~\ref{thm:weak} by showing the monotonicity property  $p^{\square}_c\paren{i,d} < p^{\square}_c\paren{i,d-1}< p^{\square}_c\paren{i+1,d}$ if $0<i<d-1,$ and corresponding result for the thresholds $q^{\square}_c$ (Proposition~\ref{prop:monotonicity}). This is done by comparing percolation on $\T_N^{d}$ with percolation on a subset homotopy equivalent to $\T^{d-1}.$ The proof of the corresponding result for permutohedral site percolation is different, but the overall idea is similar (Proposition~\ref{prop:monotonicity_permuto}).

\subsection{Examples}\label{sec:examples}
We begin by describing the homology of the cubical complex $\T_N^d.$ $\T_N^d$ is the product space $S^1\times \ldots S^1,$ where the $j$-th circle $S^1$ is the loop of length $N$ in the $j$-th coordinate direction passing through the origin. The homology of a product space is described by the K\"{u}nneth formula for homology (see Section 3b of~\cite{hatcher2002algebraic}). It tells us that products of these circles generate the homology of the torus, in the sense that a basis for $H_k\paren{\T^d_N; \F}$ is given by the $\binom{d}{k}$ (oriented) $k$-tori obtained from $\paren{S^1}^{k}\times \set{0}^{d-k}$ by permuting the coordinates. More formally, we can  write $H_k\paren{\T^d_N; \F}$ as the exterior product $\bigwedge^k H_1\paren{\T^d_N;\F}$: let $x_j$ denote a generator of the homology of the $j$-th coordinate circle of $\T_N^d.$ Then the the elements of the standard basis for $H_k\paren{\T^d_N; \F}$ are written as $x_{i_1} \wedge x_{i_2} \wedge \ldots \wedge x_{i_k},$ where $i_1<i_2\dots < i_k.$ By anticommutativity, transposing $x_{i_l}$ and $x_{i_m}$ reverses the orientation of the homology class.

For example, the giant cycles in Figure~\ref{fig:giant1} represent $x_1$ (it is homologous to the circle $S^1\times \set{0},$ when oriented appropriately) and $x_2$ (corresponding to $\set{0}\times S^1$) respectively. Note that, in both cases, the complement contains cycles in the same homology classes of the ambient torus. Similarly, the standard generators $x_1,x_2,x_3$ for $H^1\paren{\T^3;\F}$ are represented by the oriented circles $S^1\times\set{0}\times\set{0},$ $\set{0}\times S^1\times\set{0},$ and $\set{0}\times\set{0}\times S^1.$ On the other hand, when $i=2$ and $d=3,$ a basis is given by the homology classes corresponding to the  oriented two-tori $S^1\times S^1 \times\set{0}$ ($x_1x_2$) $S^1 \times \set{0}\times S^1,$ ($x_1\wedge x_3$) and $\set{0}\times S^1 \times S^1$ ($x_2 \wedge x_3$.) These appear as planes when depicted in the cube with periodic boundary conditions. The giant cycle in Figure~\ref{fig:sheet} is in the homology class $x_1 \wedge x_2.$ 

 \begin{figure}
    \centering
     \includegraphics[width=0.4\textwidth]{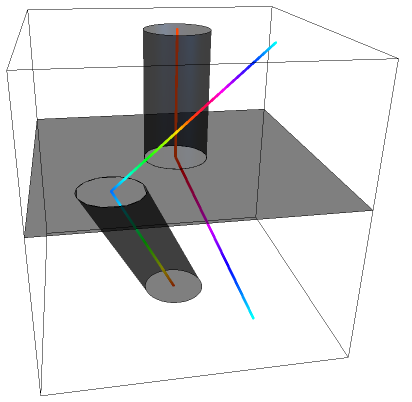}
     \qquad     
     \includegraphics[width=0.4\textwidth]{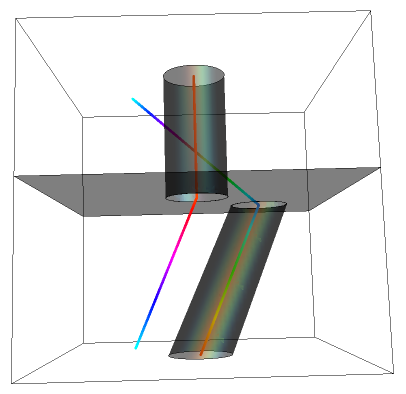}

         \caption{An embedding of $\T^2\# \R P^2$ into $\T^3,$ shown in gray from two different perspectives, with periodic boundary conditions. The rainbow path is in some sense dual to this embedded surface; note that it goes twice around $\T^3.$ See the text for more details.}
    \label{fig:CE}
\end{figure}

Consider the the configuration depicted in Figure~\ref{fig:critical}. Both $P$ and $P^{\bullet}$ include cycles homologous to $x_1+x_2$ in $H_1\paren{\T^2;\F}$ and we have $\rank\phi_*=\rank\psi_*=1.$ We can use this example to illustrate the importance of Lemma~\ref{lemma:spinning2} to our proof, and to explain why our arguments break down in characteristic $2.$ Let $B$ be the event that $x_1+x_2\in \im \phi_*.$ Then $B$ is symmetric to the event $B'=\set{g_1-g_2\in \im\phi_*}.$ When $\mathrm{char}\paren{\F}\neq 2,$ $\paren{x_1+x_2}+\paren{x_1-x_2}=2g_1\neq 0$ and $x_1\in \im\phi_*.$ This implies that $x_2$ is also in the image of $\phi_*.$ Thus $S=B\cap B'.$ That is, we can use a non-zero probability of the event $B$ to ``spin up'' basis for $H_1\paren{\T^2;\F}$ with non-zero probability. However, if $\mathrm{char}\paren{\F}=2$ then  $\paren{x_1+x_2}+\paren{x_1-x_2}=2x_1=0$ and the argument does not go through. This is related to the fact that $H_1\paren{\T^2;\F}$ is an irreducible representation of the symmetry group of $\T^2$ when $\mathrm{char}\paren{\F}\neq 2$ but it has the invariant subspace $\mathrm{span}\paren{x_1+x_2}$ when $\mathrm{char}\paren{\F}=2.$ 

The three-dimensional example in Figure~\ref{fig:CE} further illustrates the dependence of the events $A$ and $S$ on $\F.$ Here, $P$ is obtained by taking the ``plane'' $S^1\times S^1 \times \set{0}$ in $\T^3$, removing two disks, and attaching a tube to the resulting boundaries so that it wraps around the torus in the orthogonal direction. Then $P$ is a non-orientable surface of genus $2$ and is homeomorphic to the connected sum $\T^2\# \R P^2.$ When coefficients are taken in $\Z_2,$ $P$ represents the homology class $x_1 \wedge x_2$ (the difference between the chains corresponding to $P$ and  $S^1\times S^1 \times \set{0}$ is the boundary of the interior of the tube) and $\rank \im\phi_*=1$. On the other hand, the chain corresponding to $P$ is not a co-cycle in $C_2\paren{P;\Q};$ the plaquettes in the ``plane'' cannot be oriented so that their boundaries cancel out with both ends of the tube. Thus $H_2\paren{P;\Q}=0$ and $\im\phi_*=0$ if the coefficients are in $\Q.$ Now, consider the dual rainbow path in the figure. When oriented appropriately, this cycle represents $2x_3$ in $H_1\paren{P;\Q}$ but is $0$ in  $H_2\paren{P;\Z_2}.$ That is, $x_{3}\in\im \psi_*$ when coefficients are taken in $\Q$ but not $\Z_2.$ Observe that $x_{1}, x_{2}\in \im\psi_*$ regardless of the characteristic. In summary, the following events occur: $A\paren{\Z_2},$ $\neg A\paren{\Z},$ $\neg S^{\bullet}\paren{\Z_2}$ and $S^{\bullet}\paren{\Z},$ where we have indicated the coefficients in parentheses.

Finally, let us note that it is plausible that choosing $\F=\Z_2$ may lower the threshold for the event $A$, as non-orientable giant cycles could potentially appear before orientable ones. When $d=2i,$ these non-orientable cycles would have to wrap around the torus in multiple directions. Moreover, by duality, the standard generators of homology for the torus would appear at a strictly larger value of $p$ than when coefficients are taken in $\Q.$  While we have not ruled out this possibility, we would consider it surprising.

\section{Topological Results}
\label{sec:topological}

In this section, we discuss a duality lemma which will be useful in many of our arguments. The model for the plaquette version is Lemma 3.1 of~\cite{bobrowski2020homological}, which is formulated for the permutohedral lattice. We will use their notation which, for a subcomplex $X \subset {\bf T}^d_N,$ defines
$$\mathcal{B}_k\paren{X} \coloneqq \rank \varphi_*,$$
where $\varphi_*:H_k\paren{X} \to H_k\paren{{\bf T}^d_N}$ is the map on homology induced by inclusion.

We first need to show a preliminary result demonstrating a relationship between the complement of $P$ and $P^{\blacksquare},$ and to do so we make a short topological detour and recall the notion of homotopy (see Chapter 0 of~\cite{hatcher2002algebraic}). Let $X$ and $Y$ be two topological spaces and let $f,g : X \to Y$ be continuous functions. A \emph{homotopy} between $f$ and $g$ is a continuous function $I : X \times \brac{0,1} \to Y$ satisfying $I\paren{x,0} = f(x)$ and $I\paren{x,1} = g((x)$ for every $x \in X.$ Two such functions are called homotopic. We then say that $X$ and $Y$ are \emph{homotopy equivalent} and write $X \simeq Y$ if there are continuous functions $f : X \to Y$ and $H : Y \to X$ such that $f \circ h$ is homotopic to $\mathrm{id}_Y$ and $h \circ f$ is homotopic to $id_Y.$ A \emph{deformation retraction} of $X$ onto a subspace $A$ is a homotopy between a function $f : X \to A$ that fixes $A$ and $\mathrm{id}_A.$ Note that a deformation retraction always gives a homotopy equivalence.

For the purposes of this article, one can think of a homotopy equivalence as a weaker version of a homeomorphism that preserves information such as homology groups. We only work with homotopies in the following lemma and in Section~\ref{sec:permutohedra}.

\begin{Lemma}\label{lemma:deformationretract}
$\T^d\setminus P$ deformation retracts to $P^{\blacksquare}.$ 
\end{Lemma}
\begin{proof}

Let $T^{(j)}$ and $T^{(j)}_{\blacksquare}$ denote the  $j$-skeletons of $\T^d_N$ and  $\paren{\T^d_N}^\blacksquare$, respectively, and let 
\[S_j=T^{(d-j)}_{\blacksquare}\setminus T^{(j)}\,.\]
Observe that  $S_j$ is obtained from $T^{(d-j)}_\blacksquare$ by removing the central point of each $(d-j)$-cell. Also, let
\[\hat{S}_j=T^{(d-j)}_{\blacksquare}\setminus P\,.\] 
We construct a deformation retraction from $\T^d\setminus P=\hat{S}_0$ to $P^{\blacksquare}$ by iteratively collapsing $\hat{S}_j$ to $\hat{S}_{j+1}$ for $j<i,$ then collapsing $\hat{S}_{i}$ to $P^{\blacksquare}.$ 

For an $j$-cell $\sigma$ of $T^{(j)}_\blacksquare$ with center point $q$ let 
\[f_{\sigma}:\sigma\setminus \set{q}\times\brac{0,1}\rightarrow\sigma\setminus\set{q}\]
be the deformation retraction from the punctured $j$-dimensional cube to its boundary along straight lines radiating from the center. Observe that the restriction of $f_\sigma$ to $\paren{\sigma\setminus P} \times \brac{0,1}$ defines a deformation retraction from $\sigma\setminus P$ to $\partial\sigma\setminus P$ (for $j>d-i$); this is because $\sigma$ intersects $P$ in hyperplanes spanned by the coordinate vectors based at $q.$ When projecting radially from $q,$ points inside $\sigma\cap P$ remain inside $\sigma\cap P$ and points outside of $\sigma\cap P$ remain outside of $\sigma\cap P.$

\begin{figure}[t]
    \centering
    \includegraphics[width=.5\textwidth]{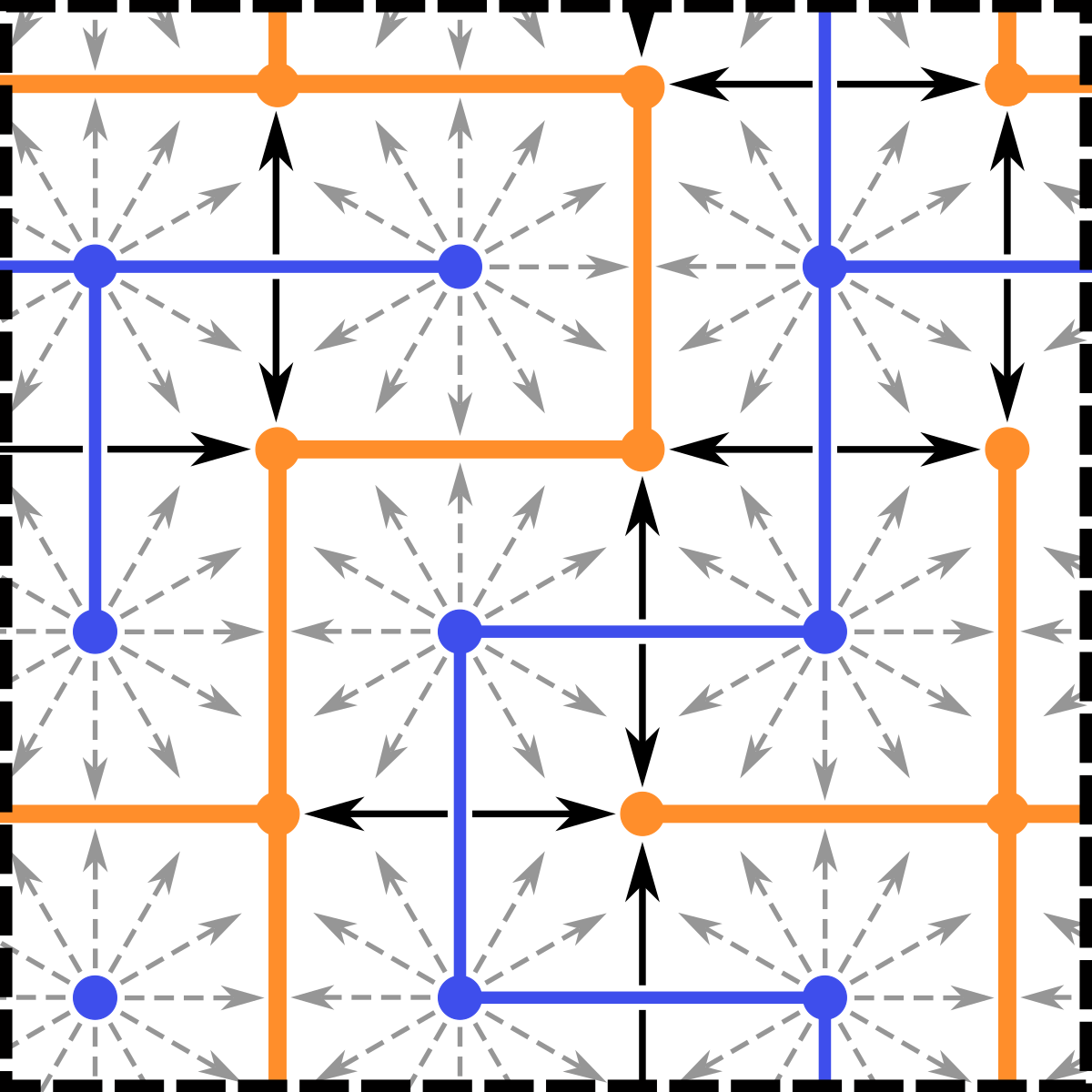}
    \caption{A part of the of the deformation retraction for the case $N=3$, $d=2,i=1.$ $P$ is shown in blue and $P^{\blacksquare}$ in orange. $\T^d\setminus P$ is first retracted to $T^{(1)}_\blacksquare\setminus P$ via the dashed gray arrows radiating from each vertex of $P,$ then to $P^{\blacksquare}$ by the solid black arrows radiating from the midpoints of the edges of $P.$}
\end{figure}

For $x\in S_j,$ let $\sigma\paren{x}$ be the unique  $(d-j)$-cell of $(\T^d_N)^\blacksquare$ that contains $x$ in its interior. Define $G_j:S_j\times\brac{0,1}\rightarrow S_j$ by 
\[G_j\paren{x,t}=\begin{cases}
f_{\sigma\paren{x}}\paren{x,t} & x\in S_j\setminus T^{d-j-1}_{\blacksquare}\\ x & \text{otherwise}\,.\end{cases}\]
$G_j$ collapses $S_j$ to $T^{(d-j-1)}$ by retracting the punctured $(d-j)$-cells to their boundaries. It follows from the discussion in the previous paragraph that the restriction of $G_j$ to $\hat{S}_j\times\brac{0,1}$ defines a deformation retraction from $\hat{S}_j$ to  $\hat{S}_{j+1}.$ 

Similarly, define $H:\hat{S}_{i}\times\brac{0,1}\rightarrow \hat{S}_{i}$ by 
\[H\paren{x,t}=\begin{cases}
f_{\sigma\paren{x}}\paren{x,t} & x\in \hat{S}_i\setminus P^{\blacksquare}\\ x & \text{otherwise}\,.\end{cases}\]
That is, $H$ collapses the $(d-i)$-faces of $T^{(d-i)}_{\blacksquare}$ that are punctured by $i$-faces of $P$ to deformation retract $\hat{S}_i$ to  $P^{\blacksquare}.$

In summary, we can deformation retract $\T^d\setminus P$ to $P^{\blacksquare}$ via the function $F:\T^d\setminus P\times\brac{0,i}\rightarrow T^d\setminus P$ defined by
\[F\paren{x,t}=\begin{cases}
G_0\paren{x,t} & t\in\left[0,1\right]\\
G_j\paren{F\paren{x,j},t-j}  & t\in\left(j,j+1\right], 0<j<i\\
H\paren{F\paren{x,i},t-i} & t\in\left(i,i+1\right]\,.
\end{cases}
\]
\end{proof}

In fact, the same deformation retraction works when $P$ is slightly thickened, which will be useful for the next Lemma. Let $P_\epsilon$ denote the $\epsilon$-neighborhood
\[P_{\epsilon}= \{x \in \T^d_N : d(x,P) \leq \epsilon\}\,.\]

\begin{Corollary}\label{corollary:epsilonretract}
For any $0 < \epsilon < 1/2,$ the closure $\overline{\paren{\T^d \setminus P_\epsilon}}$ deformation retracts to $P^{\blacksquare}.$
\end{Corollary}
\begin{proof}
Consider the deformation retraction as in Lemma~\ref{lemma:deformationretract} restricted to $\overline{\paren{\T^d \setminus P_\epsilon}}.$ When a punctured $j$-cell $\sigma$ is retracted via $f_\sigma,$ the property that points outside of $\sigma \cap P_\epsilon$ remain outside of $\sigma \cap P_\epsilon$ is preserved even though $\sigma \cap P_\epsilon$ now is a union of thickened hyperplanes. The deformation retractions $G_j$ and $H$ are defined by collapsing different cells via the functions $f_\sigma,$ so the restricted retraction does not pass through $P_\epsilon.$
\end{proof}

The next result is a key topological tool we use in many of our arguments. It is very similar to results of~\cite{bobrowski2020homological} and~\cite{bobrowski2020homological2} for other models of percolation on the torus including Lemma~\ref{lemma:permutoduality} below. For convenience, let 
\[D=\rank H_i\paren{\T^d}=\binom{d}{i}\,.\]

\begin{Lemma}[Duality Lemma]
\label{lemma:duality}
$\rank \phi_*+\rank\psi_*=D.$ In particular, at least one of the events $A^{\square}$ and $A^{\blacksquare}$ occurs, $S^{\blacksquare} \iff Z^{\square},$ and  $Z^{\blacksquare}\iff S^{\square}.$
\end{Lemma}

\begin{proof}
We proceed in similar fashion to Bobrowski and Skraba's proof of Lemma~\ref{lemma:permutoduality}. It would be impractical to attempt to give a fully self-contained proof, so we instead provide references to all of the technical tools required. Let $\epsilon = 1/4$ and define $P_\epsilon^c \coloneqq \overline{\paren{\T^d_N \setminus P_\epsilon}}.$ Consider the commutative diagram
\[\begin{tikzcd}[cramped, sep=small]
H_i\paren{P_\epsilon} \arrow{r}{i_*}  & H_i\paren{\T^d_N} \arrow{r} & H_i\paren{\T^d_N,P_\epsilon} \arrow{r}{\delta_i} & H_{i-1}\paren{P_\epsilon}\\%
H^{d-i}\paren{\T^d_N,P_\epsilon^c} \arrow{r} \arrow[swap]{u}{\cong} & H^{d-i}\paren{\T^d_N} \arrow{r}{j^*} \arrow[swap]{u}{\cong} & H^{d-i}\paren{P_\epsilon^c} \arrow{r}{\delta^{d-i}} \arrow[swap]{u}{\cong} & H^{d-i+1}\paren{\T^d_N,P_\epsilon^c} \arrow[swap]{u}{\cong}
\end{tikzcd}
\]

where $i$ and $j$ are the inclusions of $P_\epsilon$ and $P^c_\epsilon$ respectively into $\T_N^d$.  The group $H_k\paren{X,Y}$ for a subspace $Y \subset X$ is called a relative homology group, which we will not define here but is described in Section 2.1 of~\cite{hatcher2002algebraic}. When $X$ and $Y$ are sufficiently nice (as is the case here), $H_k\paren{X,Y}$ is the homology of the quotient space $X/Y$ obtained by identifying all points of $Y.$ Also, the groups $H^k\paren{X}$ and $H^k\paren{X,Y}$ denote cohomology and relative cohomology groups, respectively. Since we are using field coefficients, these are just the dual vector spaces $\mathrm{Hom}\paren{H^k\paren{X},\mathbb{F}}$ and $\mathrm{Hom}\paren{H^k\paren{X,Y},\mathbb{F}}.$

Each row of the diagram is a part of the long exact sequence for a pair $(X,Y),$ a key tool relating the homology groups of a subspace with those of the ambient space. Here ``exact'' means that the image of each map is the kernel of the next. This sequence is also covered in  Section 2.1 of~\cite{hatcher2002algebraic}. In order to see why the vertical maps are isomorphisms, we use classical tools from algebraic topology that relate the $k$-th homology to the $(d-k)$-th cohomology in a $d$-dimensional manifold. The isomorphisms $H_i\paren{\T^d} \cong H^{d-i}\paren{\T^d}$ and $H_{i-1}\paren{\T^d} \cong H^{d-i+1}\paren{\T^d}$ are from Poincar\'{e} Duality, a fundamental symmetry of the (co)homology of a closed manifold (Theorem 3.30 of~\cite{hatcher2002algebraic}). The isomorphisms $H_i\paren{P_{\epsilon}} \cong H^{d-i}\paren{\T^d_N,P_{\epsilon}^c}$ and $H_{i-1}\paren{P_{\epsilon}} \cong H^{d-i+1}\paren{\T^d_N,P_{\epsilon}^c}$ are from a generalization of Poincar\'{e} duality to manifolds with boundary called Lefschetz duality (Theorem 3.43 in~\cite{hatcher2002algebraic}). Lastly, whenever we have a commutative diagram with each row being exact and four vertical isomorphisms arranged as above, a general fact about commutative algebra called the Five Lemma (found at the end of Section 2.1 of~\cite{hatcher2002algebraic}) implies that the final vertical map is also an isomorphism.

Since the rows come from exact sequences, one can verify from a diagram chase that \[H_i(\T^d_N) \cong \im i_* \oplus \im j^*\,.\] Furthermore, since we are considering homology with field coefficients, $\rank j^* = \rank j_*.$ Now by Corollary~\ref{corollary:epsilonretract}, $\overline{\paren{\T^d_N \setminus P_\epsilon}}$ retracts to $P^{\blacksquare},$ and $P_\epsilon$ clearly retracts to $P,$ so $\rank \phi_* = \rank i_*$ and $\rank \psi_* = \rank j^*.$ Putting these together gives $\rank \phi_* + \rank \psi_* = D.$
\end{proof}

\section{Surjectivity} 
\label{sec:surjective}
The goal of this section is to show that if $p>p^{\square}_c$ then $\mathbb{P}_p\paren{S^{\square}}\rightarrow 1$ as $N\rightarrow\infty.$
First, we will prove that $\mathbb{P}_p\paren{S^{\square}}\geq b_0 \mathbb{P}_p\paren{A^{\square}}^{b_1}$ for some $b_0,b_1>0$ that do not depend on $N.$ 

Recall that a vector space $V$ that is acted on by a group $G$ is called an \emph{irreducible representation} of $G$ if it has no proper, non-trivial $G$-invariant subspaces. That is, the only subspaces $W$ of $V$ so that $\set{gw:g\in G,w\in W}=W$ are $\set{0}$ and $V.$

\begin{Lemma}\label{lemma:spinning2}
Let $V$ be a finite dimensional vector space and $Y$ be a set. Let $\mathcal{A}$ be the lattice of subspaces of $V.$ Suppose $f : \mathcal{P}\paren{Y} \to \mathcal{A}$ is an increasing function, i.e. if $A \subset B$ then $f\paren{A} \subset f\paren{B}.$ Let $G$ be a finite group which acts on both $Y$ and $V$ whose action is compatible with $f.$ That is, for each $g \in G$ and $D\in \mathcal{P}\paren{Y}$  $g\paren{f\paren{D}} = f\paren{gD}.$ Let $X$ be a $\mathcal{P}\paren{Y}$-valued random variable with a $G$-invariant distribution that is positively associated, meaning that increasing events with respect to $X$ are non-negatively correlated. Then if $V$ is an irreducible representation of $G$, there are positive constants $C_0, C_1$ so that  $$\mathbb{P}_p\paren{f(X) = V} \geq C_0\mathbb{P}_p\paren{f(X) \neq 0}^{C_1}\,,$$ where $C_0$ only depends on $G$ and $C_1$ only depends on $\dim V.$
\end{Lemma}

\begin{proof}
Let $A_k = \set{D\in\mathcal{P}\paren{Y}:\rank f(D) \geq k}$ and $\mathcal{W}_k=f(A_k).$ For a subspace $W$ of $V$ let $\mathrm{Stab}(W)$ denote the stabilizer of $W,$  
\[\set{g \in G : gW = W}\,,\]
and for $H \leq G$ let $$L_k(H) = \set{D:\mathrm{Stab}(f(D)) = H} \cap A_k\,.$$ Then in particular, since $A_k = \sqcup_{H \leq G} L_k(H),$ there is a subgroup $H'$ of $G$ so that
\begin{equation}
\label{eq:spin0}
\mathbb{P}_p\paren{X \in L_k(H')} \geq \frac{1}{c_G} \mathbb{P}_p\paren{X \in A_k}\,,
\end{equation}
where $c_G$ is the number of subgroups of $G.$

If $H' = G$ then 
$$L_k(H')=L_k(G)=\set{D:f(D)=V}$$
because $V$ is an irreducible representation of $G,$ and it follows that
\begin{equation}
\label{eq:spin1}
\mathbb{P}_p\paren{f(X) = V} = \mathbb{P}_p\paren{X \in L_k(H')} \geq \frac{1}{c_G} \mathbb{P}_p\paren{X \in A_k}\,.
\end{equation}
Otherwise, if $\mathrm{Stab}(W) = H'$ then the orbit $\set{g W : g\in G}$ contains $\abs{G}/\abs{H'}$ elements, where the elements of each coset of $H'$ in $G$ have the same action on $W.$ Let $\mathcal{B}$ be a collection of subspaces of $V$ that contains one element from each orbit of $\set{W \in \mathcal{W}_k: \mathrm{Stab}(W) = H'}$ so
$$f(L_k(H'))= \bigsqcup_{g \in G/H'} g\mathcal{B}.$$

Taking $B \coloneqq \set{D:f(D) \in \mathcal{B}},$ we have that
$$L_k(H') = \bigsqcup_{g \in G/H'} gB.$$
Let $g\in G\setminus H'.$  The events $B$ and $gB$ are symmetric so
$$\mathbb{P}_p\paren{X \in B}=\mathbb{P}_p\paren{X \in gB}=\frac{\abs{H'}}{\abs{G}}\mathbb{P}_p\paren{L_k(H'))}\geq \frac{1}{c_G \abs{G}}\mathbb{P}_p\paren{X \in A_k}$$
using Equation~\ref{eq:spin0}. By construction, $gB \cap B \subseteq A_{k+1}$ and the positive association property of $X$ yields

\begin{align*}
    \mathbb{P}_p\paren{X \in A_{k+1}} &\geq \mathbb{P}_p\paren{X \in B \cap g B}\\
    &\geq \mathbb{P}_p\paren{X \in B}^2\\
    &\geq \paren{\frac{1}{c_G \abs{G}}\mathbb{P}_p\paren{X \in A_k}}^2\,.
\end{align*}
Since either the preceding equation or Equation~\ref{eq:spin1} holds for all $k,$ we can conclude that there are positive constants $C_0(G,V)$ and $C_1(V)$ so that
\begin{align*}
    \mathbb{P}_p\paren{f(X) = V} = \mathbb{P}_p\paren{X \in A_{\dim V}} &\geq  C_0\mathbb{P}_p\paren{X \in A_1}^{2^{\dim V - 1}} \\&= C_0\mathbb{P}_p\paren{f(X) \neq 0}^{C_1}\,.
\end{align*}
\end{proof}

Now it suffices to check the irreducibility of the homology of the torus as a representation of the point symmetry group of the cubical complex structure, which we recall is the group of symmetries that fix the origin. The point symmetry group on the $\T_N^d$ is the same as the point symmetry group of $\Z^d.$ It is the subgroup of $O(N)$ generated by the permutations of the coordinate directions and reflections which reverse a coordinate direction. This is known as the hyperoctahedral group and can be written as $W_d = S_2 \wr S_d,$ where $\wr$ denotes the wreath product.

\begin{Proposition}\label{prop:Zdirrep}
Let $\F$ be a field, $d>0,$ and $1\leq i\leq d-1.$ $H_i(\T^d_N;\F)$ is an irreducible representation of $W_d$ if and only if $\mathrm{char}\paren{\F} \neq 2.$
\end{Proposition}

\begin{proof} 
 Recall from Section~\ref{sec:examples} that $H_k\paren{\T^d_N; \F}$ is isomorphic to the exterior product $\bigwedge^k H_1\paren{\T^d_N;\F}.$ The action of the symmetry group $W_d$ on $H_k\paren{\T^d_N; \F}$ is induced by its action on  $H_1\paren{\T^d_N; \F}.$ That is, for $v_1\wedge \ldots \wedge v_k\in  \bigwedge^k H_1\paren{\T^d_N;\F}$ and $g\in W_d,$ we have that 
  \[g \paren{v_1\wedge\ldots \wedge v_k} = g v_1\wedge\ldots\wedge g v_k.\]
 In representation theory, this is simply called the $k$-th exterior power of the $W_d$-representation $H_1\paren{\T^d_N;\F} \cong \F^{d}.$ We will give a direct proof of irreducibility  by showing that if $w\in\bigwedge^i \F^d\setminus\set{\vec{0}}$ then $\langle W_d w \rangle= \langle \bigwedge^i \F^d\rangle.$

Let $w$ be an arbitrary non-zero element of $\bigwedge^i \F^d\setminus\set{\vec{0}}.$ By dividing by the leading coefficient if necessary we may write $$w = x_{l_{1,1}}\wedge \ldots \wedge x_{l_{1,i}} + \ldots + c_m x_{l_{m,1}}\wedge \ldots \wedge x_{l_{m,i}}\,.$$
 Let $\sigma_v \in S_d \leq W_d$ be a permutation so that $$\sigma_w\paren{x_{l_{1,1}}\wedge \ldots \wedge x_{l_{1,i}}} = x_1\wedge x_2\wedge \ldots \wedge x_i.$$ For each $1 \leq k \leq d,$ let $\rho_k \in W_d$ be the reflection about the hyperplane orthogonal to $x_k,$ and let $f_k(v) \coloneqq v + \rho_k(v)$ for $v \in \bigwedge^i \F^d.$ Then 
$$f_{i+1}\paren{f_{i+2}\paren{\ldots f_d\paren{\sigma_w\paren{w}}}} = 2^{d-i} x_1\wedge x_2\wedge \ldots \wedge x_i.$$
$\mathrm{char}\paren{\F} \neq 2$ so  $2^{d-i}\neq 0$ and thus $x_1\wedge x_2\wedge \ldots \wedge x_i\in \langle W_d w \rangle.$ But then using the action of $S_d \leq W_d$ on $x_1\wedge x_2\wedge \ldots \wedge x_i,$  we can obtain a basis for $\bigwedge^i \F^d.$ Thus  $\langle W_d w \rangle=\bigwedge^i \F^d$ for any non-zero $w$ and the action of $W_d$ is irreducible.
\end{proof}

We can combine Lemma~\ref{lemma:spinning2} with the preceding propositions to obtain the following corollary.
\begin{Corollary}\label{cor:bothsurjective}
Then there are constants $C_0,C_1>0$ not depending on $N,i$ such that
\[\mathbb{P}_p\paren{S^{\square}}\geq C_0\mathbb{P}_p\paren{A^{\square}}^{C_1}\,.\]
\end{Corollary}

It is worth noting that Lemma~\ref{lemma:spinning2} is more general than some of our other tools. For example, in the case of continuum percolation studied in~\cite{bobrowski2020homological2}, this Lemma can be used to show the analogue of Corollary~\ref{cor:bothsurjective}, even in the absence of stronger duality results.  

\begin{Proposition}
\label{prop:sharp2}
Let $\set{Y_N}_{N\in\mathbb{N}}$ be a sequence of finite sets with $\abs{Y_N}\rightarrow\infty,$ each of which has a transitive action by a symmetry group $H_N.$  Also, let $R\paren{N,p}$ be the random set obtained by including each element of $Y_N$ independently with probability $p,$ and suppose there are functions $f_N:\mathcal{P}\paren{Y_N}\rightarrow V$ which satisfy the hypotheses of Lemma~\ref{lemma:spinning2} for some fixed symmetry group $G$. Assume that $G$ is a subgroup of $H_N$ for all $N$ and that the action of $H_N/G$ on $V$ is trivial. If $f_N\paren{\varnothing}=0$ and $f_N\paren{Y_N}\neq 0$ for all sufficiently large $N$ then there exists a threshold function $\lambda\paren{N}$ so that for any $\epsilon>0$
\begin{align*}
\mathbb{P}\paren{f_N\paren{R\paren{N,\lambda\paren{N}-\epsilon}}=0} & \rightarrow 1\\
\mathbb{P}\paren{f_N\paren{R\paren{N,\lambda\paren{N}+\epsilon}}=V} & \rightarrow 1 \\
\end{align*}
as $N\to\infty.$
\end{Proposition}
\begin{proof}
For a fixed value of $N,$ $\mathbb{P}\paren{f_N\paren{R\paren{N,p}} \neq 0}$ is an increasing, continuous function of $p$ with 
\[\mathbb{P}\paren{f_N\paren{R\paren{N,0}} \neq 0} =0\]
and 
\[\mathbb{P}\paren{f_N\paren{R\paren{N,1}} \neq 0} =1\]
for all sufficiently large $N.$ By the intermediate value theorem we can choose $\lambda\paren{N}$ so that for all sufficiently large $N,$
\[\mathbb{P}\paren{f_N\paren{R\paren{N,\lambda\paren{N}}}\neq 0} =1/2\,.\]
Then by Lemma~\ref{lemma:spinning2}, there exist $C_0,C_1>0$ such that 
\begin{align*}
\mathbb{P}\paren{f_N\paren{R\paren{N,\lambda\paren{N}}} = V} &\geq C_0\mathbb{P}\paren{f_N\paren{R\paren{N,\lambda\paren{N}}} \neq 0}^{C_1}\\
&= \frac{C_0}{2^{C_1}} >0\,.
\end{align*}
Choose an $\epsilon_0$ between $0$ and $\frac{C_0}{2^{C_1}}.$ Note that the event 
\[\set{f_N\paren{R\paren{N,p}} = V}\]
is increasing in $p$ and invariant under the action of $H_N.$ By assumption, $H_N$ acts transitively on $X,$ so the hypotheses of Theorem~\ref{thm:FK} are met. Let $\epsilon>0.$ Re-arranging Equation~\ref{eqn:FK} gives that \[\mathbb{P}\paren{{f_N\paren{R\paren{N,\lambda\paren{N}+\epsilon}} = V}}>1-\delta\]
when
\[\log\paren{\abs{Y_N}}>\frac{\rho\log\paren{1/\paren{2\delta}}}{\epsilon}\,.\]

On the other hand, the event $\set{f\paren{R\paren{N,p}^c}=0}$ is also increasing, so by a similar argument,
$\mathbb{P}\paren{f_N\paren{R\paren{N,\lambda\paren{N}-\epsilon}}=0}\rightarrow 1\,.$
\end{proof}

In our models of interest, this tells us that $\lambda^{\square}\paren{N}$ is a sharp threshold functions of $N$ for the appearance of all giant cycles. We can then describe the behavior of both models below $q_c$ and above $p_c.$

\begin{Corollary}\label{cor:sharp}
If $p>p^{\square}_c\paren{i,d}$
then
\[\mathbb{P}_{p}\paren{S^{\square}}\rightarrow 1\,,\]
and if $p<q^{\square}_c\paren{i,d}$
\[\mathbb{P}_{p}\paren{A^{\square}}\rightarrow 0\,.\]
\end{Corollary}

\section{The Case $d=2i$}
\label{sec:half}
We now prove Theorem~\ref{thm:half}, that $p^{\square}_c\paren{i,2i}=1/2$ is a sharp threshold for $A^{\square}$ when $d=2i.$

\begin{proof}[Proof of Theorem~\ref{thm:half}]
Half-dimensional plaquette percolation is self-dual so $\mathbb{P}_{1/2}\paren{A^{\square}}=\mathbb{P}_{1/2}\paren{A^{\blacksquare}}.$ By Lemma~\ref{lemma:duality} at least one of the events $A^{\square}$ and $A^{\blacksquare}$ must occur. Therefore,
\[2\mathbb{P}_{1/2}\paren{A^{\square}}=\mathbb{P}_{1/2}\paren{A^{\square}}+\mathbb{P}_{1/2}\paren{A^{\blacksquare}}\geq 1\]
and 
\[\mathbb{P}_{1/2}\paren{A^{\square}}\geq 1/2\]
for all $N.$ It follows that $p^{\square}_c\leq 1/2.$ Thus, if $p>1/2$ then
\[\mathbb{P}_p\paren{S^{\square}}\rightarrow 1\]
as $N\rightarrow\infty,$ and if $p<1/2$ then 
\[\mathbb{P}_p\paren{A^{\square}}\rightarrow 0\]
as $N\rightarrow\infty$ by Corollary~\ref{cor:sharp}.
\end{proof}

\section{Sharpness and Duality}
\label{sec:duality}

In this section, we combine the Duality Lemma (Lemma~\ref{lemma:duality}) with Corollary~\ref{cor:sharp} to examine the behavior of $\mathbb{P}_p\paren{A^{\square}}$ below $q^{\square}_c\paren{i,d}$ and above $p^{\square}_c\paren{i,d}.$ We also relate these thresholds to $p^{\square}_c\paren{d-i,d}$ and $q^{\square}_c\paren{d-i,d}.$ 

Now we show a partial duality result for any $i$ and $d.$
\begin{Proposition}
\label{prop:duality}
\[p^{\square}_c\paren{i,d}+q^{\square}_c\paren{d-i,d}=1\,.\]
\end{Proposition}
\begin{proof}
Let $p>p^{\square}_c\paren{i,d}.$ Then
\begin{align*}
\mathbb{P}_{p}\paren{A^{\blacksquare}}&=\;\;1-\mathbb{P}_{p}\paren{Z^{\blacksquare}}&&\text{by definition}\\
&=\;\;1-\mathbb{P}_p\paren{S^{\square}} &&\text{by Lemma~\ref{lemma:duality}}\\
&\rightarrow \;\; 0 &&\text{by Corollary~\ref{cor:sharp}}
\end{align*}
as $N\rightarrow\infty.$ Therefore,  $1-p\leq q^{\square}_c\paren{d-i,d}$ for all $p>p^{\square}_c\paren{i,d}$ and
\begin{equation}
\label{eqn:dualityProp}
p^{\square}_c\paren{i,d}+q^{\square}_c\paren{d-i,d}\geq 1\,.
\end{equation}

Until now, we have suppressed the dependence of probabilities of events on $N.$ To work with subsequences in this argument, denote the probability of an event $B$ for $P\paren{i,d,N,p}$ by $\mathbb{P}_{p,N}\paren{B}.$ 

Let $p<p^{\square}_c\paren{i,d}.$ Then there is a subsequence $\set{n_1,n_2,\ldots}$ of $\N$ for which
\[\mathbb{P}_{p,n_k}\paren{A^{\square}}\rightarrow 0\,\,.\]
By Lemma~\ref{lemma:duality},
\[\mathbb{P}_{p,n_k}\paren{S^{\blacksquare}}\rightarrow 1\]
so
\[\limsup_{N\rightarrow\infty}\mathbb{P}_{p}\paren{S^{\blacksquare}}=1\]
and $1-p\geq q^{\square}_c\paren{i,d}$ for all $p<p^{\square}_c\paren{i,d}.$ Therefore,
\[p^{\square}_c\paren{i,d}+q^{\square}_c\paren{d-i,d}\leq 1\]
which holds with equality by Equation~\ref{eqn:dualityProp}.
\end{proof}

Propositions~\ref{prop:sharp2} and~\ref{prop:duality} show that duality between $p^{\square}_c\paren{i,d}$ and $p^{\square}_c\paren{d-i,d}$ is equivalent to the convergence of the threshold function $\lambda^{\square}\paren{N}.$

\begin{Corollary}
\label{cor:sharpness}
The following are equivalent.
\begin{enumerate}[label=(\alph*)]
    \item $\lim_{N\to\infty} \lambda^{\square}\paren{N}$ exists.
    \item $p^{\square}_c \paren{i,d}=q^{\square}_c\paren{i,d}.$
    \item $p^{\square}_c \paren{i,d}+p^{\square}_c\paren{d-i,d}=1.$
\end{enumerate}
\end{Corollary}

In the next section, we demonstrate that the statements of Corollary~\ref{cor:sharpness} hold in the cases $i=1$ and $i=d-1.$

\section{The Cases $i=1$ and $i=d-1$}
\label{sec:one}

We show that $p^{\square}_c\paren{1,d}$ and $q^{\square}_c\paren{1,d}$ coincide and equal the critical threshold for bond percolation on $\Z^d,$ denoted here by $\hat{p}_c=\hat{p}_c(d).$ For $d=2,$ this is a special case of Theorem~\ref{thm:half} so we may assume that $d\geq 3.$ 

We rely on two results from the classical theory of this system in the subscritical and supercritical regimes. In the former regime, we use Menshikov's Theorem~\cite{menshikov1986coincidence}, also proven independently by Aizenman and Barsky~\cite{aizenman1987sharpness}, showing an exponential decay in the radius of the cluster at the origin. 

For a vertex $x$ and a subset $S$ of $\T^d_N$ (or $\Z^d$), denote the event that $x$ is connected to a vertex in $S$ by a path of edges in $P$ by $x\leftrightarrow S.$ 
\begin{Theorem}[Menshikov/Aizenman-Barsky]
\label{thm:menshikov}
Consider bond percolation on $\Z^d.$ If $p<\hat{p}_c$ then there exists a $\kappa\paren{p}>0$ so that
\[\mathbb{P}_p\paren{0\leftrightarrow \partial[-M,M]^d}\leq e^{-\kappa\paren{p}M}\]
for all $M>0.$ 
\end{Theorem}

In the supercritical regime, we use the following lemma of Grimmett and Marstrand~\cite{grimmett1990supercritical} on crossing probabilities inside a rectangle (the formulation we write here is given as Lemma 7.78 in~\cite{grimmett1999percolation}).

\begin{Lemma}
\label{lemma:grimmett778}
Let $d \geq 3$ and $p>\hat{p}_c.$ Then there is an $L>0$ and a $\delta>0$ so that if $N>0$ and $x\in \brac{0,N-1}^{d-1}\times \brac{0,L},$ the probability that $0$ is connected to $x$ inside  $P\cap \big(\brac{0,N-1}^{d-1}\times \brac{0,L}\big)$ is at least $\delta.$
\end{Lemma}

\begin{proof}[Proof of Theorem~\ref{thm:one}]
Consider plaquette percolation in $\T^d_N$ with $i=1.$ We will first apply Menshikov's Theorem to show that the probability of a giant one-cycle limits to zero as $N\rightarrow\infty$ when  $p<\hat{p}_c.$

 Let $p<\hat{p}_c$ and let $M=\floor{N/2}.$ For a vertex $x$ of $\mathbb{T}_N^d,$ denote by $A_x$ the event that there is a connected, giant $1$-cycle containing an edge adjacent to $x.$ If $A_0$ occurs then $0\leftrightarrow \partial[-M,M]^d$ because a $1$-cycle contained in $[-M,M]^d$ is null-homologous in $\mathbb{T}_N^d$.
 Therefore,
 \begin{equation}
     \label{eq:pboundsn1}
     \mathbb{P}_p\paren{A_x}\leq e^{-\kappa\paren{p}M}
 \end{equation}
for all vertices $x$ of $\mathbb{T}_N^d,$ using translation invariance and Theorem~\ref{thm:menshikov}. 

Let $X$ be the number of vertices in $\mathbb{T}_N^d$ that are contained in a connected, giant $1$-cycle. 
$A^{\square}=\set{X\geq 1}$ so

\begin{align*}
    \mathbb{P}_p\paren{A^{\square}}&=\;\;\mathbb{P}_p\paren{X\geq 1}\\
    &\leq\;\; \mathbb{E}_p\paren{X} && \text{by Markov's Inequality}\\
    &=\;\; \sum_{x\in \mathbb{T}_N^d}\mathbb{P}_p\paren{A_x} \\
    & \leq\;\;  N^D e^{-\kappa\paren{p}M} &&\text{using Equation ~\ref{eq:pboundsn1}}\\
    & = \;\; N^d e^{-\kappa\paren{p}\floor{N/2}}
\end{align*}
which goes $0$ as $N\rightarrow\infty$.

Now we consider the supercritical case. Let $p>p'>\hat{p}_c$ and let $B_j$ be the event that there is a path of edges of $P$ connecting $0$ to $\paren{N-1}\vec{e}_j$ inside of $\brac{0,N-1}^d$. By the previous lemma and by symmetry, there is a $\delta>0$ so that $\mathbb{P}_p\paren{B_j}\geq \delta$ for every $1 \leq j \leq d.$ 

If $B_j$ occurs, then the path obtained by adding the edge between $\paren{N-1}\vec{e}_j$  and $N\vec{e}_j=0$ is a giant $1$-cycle. As $H_1\paren{\T^d}$ is generated by the circles in the coordinate directions of $\T^d_N$, it follows from Harris' lemma that
\[\mathbb{P}_{p'}\paren{S^{\square}}\geq \paren{p'}^d\,\mathbb{P}_{p'}\paren{\bigcap_{1 \leq j \leq d} B_j} \geq \paren{p'}^d\,\mathbb{P}_{p'}\paren{B_1}^d \geq  \paren{p'}^d\delta^d\,.\]
Then by applying Theorem~\ref{thm:FK}, we see that $\mathbb{P}_p\paren{S^{\square}} \to 1$
as $N \to\infty$ as desired.

We have now proven the statement in Theorem~\ref{thm:one} regarding the case of $i=1.$ The case of $i=d-1$ follows immediately from an application of Lemma~\ref{lemma:duality}.
\end{proof}

\section{Monotonicity}
\label{sec:weak}

Next, we prove that the critical probabilities $p^{\square}_c\paren{i,d}$ are strictly increasing in $i$ and strictly decreasing in $d.$ This will complete the proof of Theorem~\ref{thm:weak}. 

Our strategy will be to compare percolation on $\T^d_N$ with the thickened $(d-1)$-dimensional slice $\T^d_N \cap \set{0 \leq x_1 \leq 1}.$ Compare the first part of the proof to that of Lemma 4.9 of~\cite{bobrowski2020homological2}.

\begin{Proposition}
\label{prop:monotonicity}
For $0<i<d-1,$
\[p^{\square}_c\paren{i,d} < p^{\square}_c\paren{i,d-1} < p^{\square}_c\paren{i+1,d}\,.\]
\end{Proposition}
\begin{proof}
First, we will show that
\begin{equation}
\label{eqn:prop20d}
p^{\square}_c\paren{i,d}\leq p^{\square}_c\paren{i,d-1} \leq  p^{\square}_c\paren{i+1,d}\,.
\end{equation}
Let $T=\T^d_{N}\cap\set{x_1=0}.$ $T$ is a torus of dimension $d-1$ and, by a standard argument, the map on homology $\alpha_*:H_j\paren{T}\rightarrow H_j\paren{\T^d}$ induced by the inclusion $T\hookrightarrow \T^d$  is injective for all $j$. $P\cap T$ is distributed as $P\paren{i,d-1,N,p}.$

Define $A^{\square}_{d-1}$ to be the event that $\gamma_*:H_{i}\paren{P\cap T} \rightarrow H_i\paren{T}$ is non-zero, where $\gamma_*$ is induced by the inclusion $P\cap T \hookrightarrow T.$ If $A^{\square}_{d-1}$ holds then $\alpha_*\circ \phi_*$ is also non-zero, as $\alpha_*$ is injective. But $\alpha_*\circ\gamma_{*}=\phi_*\circ\beta_*,$ where $\beta_*$ is the map on homology $\beta_*:H_{i}\paren{P\cap T}\rightarrow H_i\paren{P}$ induced by the inclusion $P\cap T\hookrightarrow P$, so $\phi_*$ is also non-zero. It follows that $A^{\square}_{d-1}\implies A^{\square}.$ Therefore $p^{\square}_c\paren{i,d-1}\geq p^{\square}_c\paren{i,d}$ by the definition of that threshold. 

Observe that $H_i\paren{\T^d}$ is generated by the images of the maps on homology $H_i\paren{\T^d\cap\set{x_j=0}}\rightarrow H_i\paren{\T^d}$ induced by the inclusions $\T^d\cap\set{x_j=0}\hookrightarrow \T^d$ as $j$ ranges from $1$ to $d.$ Denote by $S_j$ the event that the map  $H_i\paren{P\cap\set{x_j=0}}\rightarrow H_i\paren{\T^d\cap\set{x_j=0}}$ induced by inclusion is surjective and let $q>q^{\square}_c(i,d-1).$ Then there is a subsequence $\paren{n_1,n_2,\ldots}$ of $\N$ so that
\[\mathbb{P}_{p,n_k}\paren{S_j}\rightarrow 1\]
as $k\rightarrow\infty$ for $j=1,\ldots,d.$ As $S\subset \cap_{j} S_j,$  Harris's Inequality implies that $\mathbb{P}_{p,n_k}\paren{S}\rightarrow 1$ also. Therefore, $p>q^{\square}_c\paren{i,d}$ and $q^{\square}_c\paren{i,d-1}\geq q^{\square}_c\paren{i,d}$ for all $i$ and $d.$ Combining this inequality (for a different choice of $i$ and $d$) with Proposition~\ref{prop:duality} we obtain
\begin{equation}p^{\square}_c(i,d-1) = 1-q^{\square}_c(d-i-1,d-1) \leq 1-q^{\square}_c(d-i-1,d) = p^{\square}_c(i+1,d)\,,
\label{eqn:prop20e}
\end{equation}
which shows Equation~\ref{eqn:prop20d}.

It will be useful later in the argument to observe that these inequalities, together with Theorem~\ref{thm:one} and known lower bounds on $\hat{p}_c$ (see~\cite{BR06}, for example), imply that
\begin{equation}
\label{eqn:prop20a}
    0<p^{\square}_c\paren{1,d} \leq p^{\square}_c\paren{i,d}\leq p^{\square}_c\paren{i,i+1}<1\,.
\end{equation}

Furthermore, we can show $p^{\square}_c\paren{i,d} < p^{\square}_c\paren{i,d-1}$ using the thicker cross-section $T' = \T^d_N \cap \set{0 \leq x_1 \leq 1}.$ Note that an $i$-face $v$ of $T$ is in the boundary of a unique $(i+1)$-face $w(v)$ of $T'$ that is not contained in $T$ (for example, if $v=\set{0}\times \brac{0,1}^i \times \set{0}^{d-i-1},$ then $w(v)=\brac{0,1}^{i+1} \times \set{0}^{d-i-1}$). The idea is to sometimes add $v$ to $T$ when the other $i$-faces of $w(v)$ are present, effectively increasing the percolation probability in $T$ by a small amount. However, we must be careful to do so in a way so that the $i$-faces remain independent.

The $i$-faces of $T'$ are divided into three subsets: those included in $T,$ those which are perpendicular to $T$ (that is, $i$-faces not included in $T$ which intersect $T$ in their boundary), and those parallel to $T$ (that is, $i$-faces of the form $v+\vec{e}_1$, where $v$ is an $i$-face of $T$). For an $i$-face $v$ of $T,$ let $J(v)$ be the set of all perpendicular $i$-faces that meet $v$ at an $i-1$ face. $v,$ $v+\vec{e}_1,$ and $J(v)$ are the $i$-faces of the $(i+1)$-face $w(v).$ Also, for a perpendicular $i$-face $u$ of $T',$ let $K(u)=\set{v:u\in J(v)}.$ Note that for any $u$ and $v$
\begin{equation}
\label{eqn:prop20b}
\abs{J(v)}=2i \qquad \text{and} \qquad \abs{K(u)}=2(d-i)\,.
\end{equation}

\begin{figure}[t]
    \centering
    \includegraphics[width=.8\textwidth]{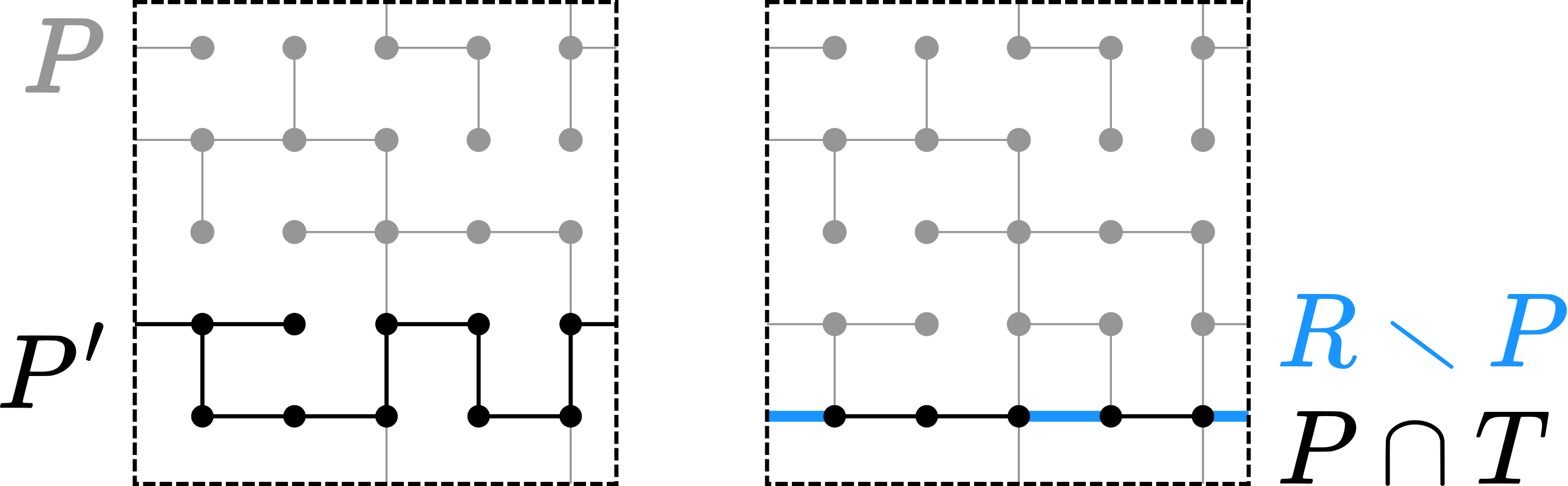}
    \caption{\label{fig:prop20setup} The setup in the proof of Proposition~\ref{prop:monotonicity} for the case $d=2,i=1.$ On the left, $P'$ is shown in black and the remaining faces of $P$ are depicted in gray. On the right, $P\cap T$ is in black and the additional faces of $R$ are shown in blue. Note that giant cycles exist in $P, P',$ and $R,$ but not in $P\cap T.$}
\end{figure}

We define a coupling between $i$-dimensional plaquette percolation $P'$ on $T'$ with probability $p$ and $i$-dimensional percolation $R$ with probability $p+p(1-p)q^{2i}$ on $T,$ where $q=q(p)$ is chosen to satisfy $p = 1-(1-q)^{2(d-i)}$.  For all pairs $\paren{v,u}$ where $v$ is an $i$-face of $T$ and $u\in J(v),$ define independent Bernoulli random variables $\kappa\paren{u,v}$ to be $1$ with probability $q$ and $0$ with probability $1-q.$ Let $P'\subset T'$ be the subcomplex containing the $(i-1)$-skeleton of $T'$ where each $i$-face in $T$ or parallel to $T$ is included independently with probability $p,$ and the other $i$-faces $u$ of $T'$ are included if $\kappa\paren{u,v}=1$ for at least one $v\in K\paren{u}.$ Observe that 
\[\mathbb{P}\paren{u\in P'}=1-\mathbb{P}\paren{\cap_{v \in K\paren{u}} \set{\kappa\paren{u,v}=0}} = 1-(1-q)^{2(d-i)}=p\]
(using Equation~\ref{eqn:prop20b}), and that the faces $u$ are included independently. That is, $P'$ is percolation with probability $p$ on $T'.$ On the other hand, define $R\subset T$ by starting with all faces of $P'\cap T$ and adding an $i$-face $v\notin P'$ if $v+\vec{e}_1\in P'$ and $\kappa\paren{v,u}=1$ for all $u\in J\paren{v}.$ Then $R$ is percolation on $T$ with probability 
$p+p(1-p)q^{2i} > p.$ See Figure~\ref{fig:prop20setup}.

As $p+p(1-p)q^{2i}$ is a continuous function of $p$ and $0<p^{\square}_c(i,d-1) < 1$ (Equation~\ref{eqn:prop20a}), we can choose $p$ to satisfy 
\[0<p< p^{\square}_c(i,d-1)<p+p(1-p)q^{2i}<1\,.\] Then
\begin{equation}
\label{eqn:prop20c}
\mathbb{P}_{p+p(1-p)q^{2i}}\paren{\text{$\xi_*$ is non-trivial}}\to 1
\end{equation}
as $N\to \infty$ by Corollary~\ref{cor:sharp}, where $\xi_*:H_i(R)\rightarrow H_i(T)$ is the map on homology induced by the inclusion $R\hookrightarrow T.$

Extend $P'$ to plaquette percolation $P$ on all of $\T^d_N$ by including the $i$-faces in $\T^d_N\setminus T'$ independently with probability $p.$ If $\sigma$ is an $i$-cycle of $R$ we can write
\[\sigma=\sum_{j} a_j u_j + \sum_{k} b_k v_k \]
where $u_j\notin P$ and $v_k\in P$ for all $j$ and $k.$ Then, by construction, we can form a corresponding  $i$-cycle $\sigma'$ of $P$ by setting
\[\sigma'=\sigma+\sum_j a_j \partial w\paren{u_j}\,.\]
 $\sigma$ and $\sigma'$ are homologous in $\T^d,$ so $\alpha_*\circ \xi_*\paren{\brac{\sigma}}=\phi_*\paren{\brac{\sigma'}}$  In particular, if $\xi_*$ is non-trivial then $\phi_*$ is non-trivial as well. Using Equation~\ref{eqn:prop20c}, it follows that
\[\mathbb{P}_{p}\paren{A^{\square}}\geq \mathbb{P}_{p+p(1-p)q^{2i}}\paren{\text{$\xi_*$ is non-trivial}} \to 1\,,\]
 as $N\to \infty.$ Therefore,
\[p^{\square}_c\paren{i,d}\leq p < p^{\square}_c\paren{i,d-1}\,.\]

We can define similar couplings between percolations on $\T^d_N \cap \set{x_j = 0}$ and on $\T^d_N \cap \set{0 \leq x_j \leq 1}$ for $j = 1 \ldots d.$ Combining these couplings with the argument leading to Equation~\ref{eqn:prop20e} yields $q^{\square}_c(i,d) < q^{\square}_c(i,d-1).$ Then from Proposition~\ref{prop:duality} we obtain
\[p^{\square}_c(i,d-1) = 1-q^{\square}_c(d-i-1,d-1) < 1-q^{\square}_c(d-i-1,d) = p^{\square}_c(i+1,d)\,.\]
\end{proof}

\begin{figure}[t]
    \centering
    \includegraphics[scale=.75]{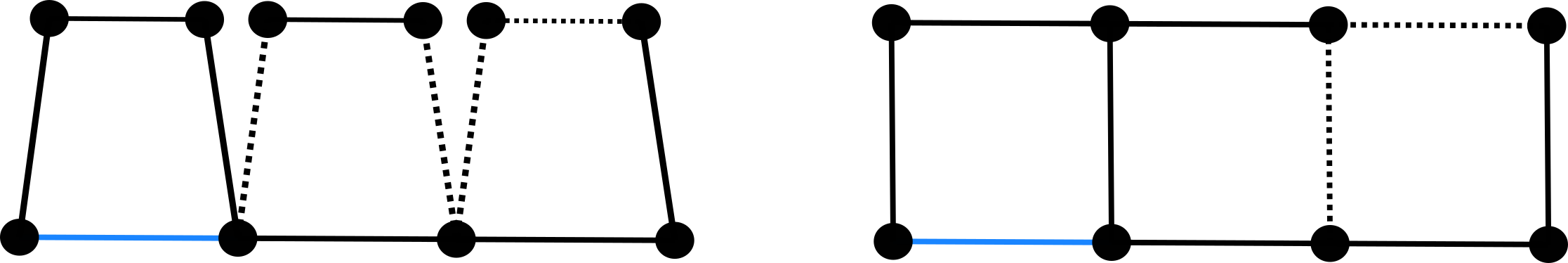}
    \caption{\label{fig:strict} Percolation on $T''$ (left) mapping to percolation on $T'$ (right) for $i=1,d=2.$ The blue edges are in $R \setminus P.$}
\end{figure}

An alternative approach to the proof of Proposition~\ref{prop:monotonicity} is to construct a third space $T''$ by attaching a new $(i+1)$-cube to each $i$-face of $T$ along one of the cube's $i$-faces. We can define inhomogeneous percolation $P''$ on $T''$ by starting with the $(i-1)$-skeleton of $T'',$ adding each $i$-face of $T$ and and each $i$-face parallel to $T$ independently with probability $p,$ and adding the perpendicular $i$-faces independently with probability $q$ (these faces play the same role as the random variables $\kappa\paren{u,v}$ above). Giant cycles in $P''$ are ones that are mapped non-trivially to $H_i\paren{T''}$ by the map on homology induced by the inclusion $P''\hookrightarrow T'',$ and they appear at a lower value of $p$ than $p^{\square}_c\paren{i,d-1}$ (precisely when they appear in $R$ as defined above). The proof is finished by observing that the quotient map $\pi:T''\rightarrow T'$ identifying the corresponding perpendicular faces of neighboring cubes induces an injective map on homology, and therefore the existence of giant cycles in $P''$ implies the existence of giant cycles in $P.$ This idea is illustrated in Figure~\ref{fig:strict}. Note that our definition of giant cycles in $T''$  can be adapted to give a more general notion of homological percolation in the $i$-skeleton of a cubical or simplicial complex whose $i$-dimensional homology is nontrivial.

\begin{proof}[Proof of Theorem~\ref{thm:weak}]
By Equation~\ref{eqn:prop20a}, $q^{\square}_c\paren{i,d},p^{\square}_c\paren{i,d}\in\paren{0,1}.$ The remaining statements follow from Corollary~\ref{cor:sharp} and Propositions~\ref{prop:duality} and~\ref{prop:monotonicity}.
\end{proof}

Note that we could alternatively show that $p^{\square}_c\paren{i,d},q^{\square}_c\paren{i,d}\in\paren{0,1}$ by modifying the first section of the proof of Theorem~\ref{thm:one} to work for the lattice of $i$-plaquettes in $\Z^d$ and using a Peierls-type argument to obtain the bound
\[\frac{1}{2d-i+1} \leq q^{\square}_c\paren{i,d} \leq p^{\square}_c\paren{i,d} \leq 1-\frac{1}{d+i+1}\,.\]

\begin{proof}[Proof of Theorem~\ref{thm:permutohedral}]
Theorem~\ref{thm:permutohedral} follows from the proofs of Theorems~\ref{thm:half},~\ref{thm:one}, and~\ref{thm:weak} with the adjustments for the permutohedral lattice noted throughout the paper. In particular, Lemma~\ref{lemma:duality} and Propositions~\ref{prop:Zdirrep} and~\ref{prop:monotonicity} are replaced by Lemma~\ref{lemma:permutoduality} and  Propositions~\ref{prop:permutoirrep} and\ref{prop:monotonicity_permuto} respectively.
\end{proof}

\section{Permutohedral Site Percolation}\label{sec:permutohedra}
Most of our previous arguments can be applied to the permutohedral lattice with minimal modification. Here we highlight the major differences from each section. 

First, since we are not dealing with cubical complexes, we will need a more general notion.
\begin{Definition}[Definition 2.38 of~\cite{kozlov2008combinatorial}]
A \emph{polyhedral complex} in $\R^n$ is a collection $X$ of convex polytopes in $\R^n$ (of different dimensions) satisfying the following two properties.
\begin{itemize}
    \item  If $\sigma\in X$ and $\tau$ is a face of $\sigma$ then $\tau\in X.$ 
    \item  If $\sigma_1,\sigma_2\in X$ and $\sigma_1\cap\sigma_2\neq \varnothing$ then $\sigma_1\cap\sigma_2$ is a face of $X.$ 
\end{itemize}
\end{Definition}
It turns out to be convenient to treat $X$ as a collection of polytopes rather than as their union. The cubical complexes considered above are a special case of a polyhedral complex all of whose faces are cubes. Another special case is that of simplicial complexes, all of whose cells are simplices (recall that a $k$-simplex is the convex hull of $(k+1)$ points in Euclidean space, assuming the points are in general position). There is also a related notion of an \emph{abstract simplicial complex}, which we require below. Formally, an abstract simplicial complex is a family of sets which is closed under taking subsets (see Chapter 10 of~\cite{edelsbrunner2014short} for an elementary account or Section 9.1 of~\cite{kozlov2008combinatorial} for a more in depth exposition). The combinatorial structure of a simplicial complex can be encoded as an abstract simplicial complex by taking the family consisting of all vertex sets of simplices in the complex.

The homology of a polyhedral complex is defined using a boundary map taking a $k$-dimensional cell to the signed sum of its $(k-1)$-faces. The details of this construction can be found in Chapter 3 of~\cite{kozlov2008combinatorial}; they are analogous to those for the special case of cubical complexes. It is a standard result that homotopy equivalence preserves homology groups, regardless of the type of complex.

Next, the definition of the dual system is simpler for the permutohedral lattice. We define
$$Q^{\bhexagon} = Q^{\bhexagon}\paren{d,N,p} \coloneqq \overline{Q^c},$$
i.e. the polyhedral complex consisting of all permutohedra that are not included in $Q$ and their lower-dimensional faces. As mentioned before, a duality result for this setting was proven in~\cite{bobrowski2020homological}.

\begin{Lemma}[Bobrowski and Skraba]\label{lemma:permutoduality}
For $0 \leq k \leq d,$
$$\mathcal{B}_k\paren{Q} + \mathcal{B}_{d-k}\paren{Q^c} = \rank H_k\paren{\T^d}.$$
\end{Lemma}

Since $d$-dimensional permutohedra always meet along $(d-1)$-dimensional faces, it is not difficult to see that $\mathcal{B}_k\paren{Q^c} = \mathcal{B}_k\paren{Q^{\bhexagon}}.$

Sections~\ref{sec:half} and~\ref{sec:duality} need essentially no changes, as does Section~\ref{sec:one} once Theorem~\ref{thm:menshikov} is replaced with the analogous result on site percolation.

We are then left with Sections~\ref{sec:surjective} and~\ref{sec:weak}, both of which require adaptation. 

\subsection{Symmetries of $\mathcal{A}_d^*$}
In order to adapt Section~\ref{sec:surjective} to the permutohedral lattice, we will need an analogue of Proposition~\ref{prop:Zdirrep}. This is because the point symmetry group of $\mathcal{A}_d^{*}$ is not $W_d,$ but instead the symmetric group $\mathcal{S}_{d+1}.$

\begin{Proposition}\label{prop:permutoirrep}
Let $\F$ be a field, $d>0,$ and $1\leq k\leq d-1.$ $H_k({\bf T}^d_N;\F)$  is an irreducible representation of $\mathcal{S}_{d+1}$ if $\mathrm{char}\paren{\F}=0$ or $\mathrm{char}\paren{\F}$ is an odd prime that does not divide $d+1.$
\end{Proposition}

\begin{proof}
First, we review the action of $\mathcal{S}_{d+1}$ on  $\mathcal{A}_d^*.$ The lattice $\mathcal{A}_d^* \subset \F^{d+1}$ has a basis  
\begin{align*}
    B \coloneqq \set{\vec{1}-d\vec{e}_k: 1\leq k \leq d}.
\end{align*}
Now $\mathcal{S}_{d+1}$ acts on $\F^{d+1}$ by permuting the coordinates, and this restricts to an action on  $\mathcal{A}_d^*$ which permutes the elements of $B\cup \set{\vec{1}-d \vec{e}_{d+1}}.$ The $\F$-vector space generated by $\mathcal{A}_d^*$ is called the standard representation of $\mathcal{S}_{d+1}.$ Denote it by $\hat{\F}^{d}.$ Analogously to the square lattice, the actions of $\mathcal{S}_{d+1}$ by lattice symmetries on $H_k\paren{{\bf T}^d_N; \F} \cong \bigwedge^k \hat{\F}^d$ are given by the exterior powers of the standard representation.

The statement follows from more general results in the literature  on the representation theory of the symmetric group. These are usually stated using language that is beyond the scope of this paper. We describe the required results from~\cite{james2006representation} and sketch why they suffice. Please refer to that reference for an explanation of the terminology used in the remainder of this proof.

Let $\lambda$ be a partition of $\set{1,\ldots,d+1}.$ As explained in Section 4 of~\cite{james2006representation}, $\lambda$ corresponds to a representation $S^{\lambda}$ of $\mathcal{S}_{d+1}$ over any field, called a Specht module. In particular, $H_k\paren{{\bf T}^d_N; \F} \cong \bigwedge^k \hat{\F}^d$ is isomorphic to the Specht module corresponding to the partition $\lambda=\paren{d-k+1,1,\ldots,1}$ of $\set{1,\ldots,d+1}.$ Theorem 4.9 of~\cite{james2006representation} implies that if $\mathbb{F}\subset \mathbb{F}'$ and $S^{\lambda}$ is irreducible over $\mathbb{F}$ then it is irreducible over $\mathbb{F}'.$ In particular, it suffices to consider the cases $\F=\Q$ and $\F=\Z_p.$ When $\F=\Q,$ the Specht modules are exactly the irreducible representations of the symmetric group by Theorem 4.12 of~\cite{james2006representation}. This implies the first case of the proposition. The second case follows from Theorem 23.7 in the same reference.

\end{proof}

\subsection{Monotonicity}\label{sec:monotonicity}
Now we obtain analogues of the results in Section~\ref{sec:weak} in the permutohedral setting. The idea of the proof is again to find a copy of $\T^{d-1}$ within $\T^d,$ but the existence of a such a sublattice is less immediately apparent than in the plaquette case. Let $\mathcal{O}$ be a finite collection of permutohedra of $A_d^*$ and let $O=\cup_{\theta \in \mathcal{O}} \theta.$

We describe how the topology of $O$ is encoded in the adjacency graph of $\mathcal{O}.$ Let $G$ be a graph with vertex set $\set{v_1,\ldots,v_n}.$ The \emph{clique complex} or \emph{flag complex} of $G$ is the abstract simplicial complex  $\Delta\paren{G}$ so that $\paren{v_{i_1},\ldots,v_{i_l}}\in \Delta\paren{G}$ if and only $v_{i_j}$ is adjacent to  $v_{i_l}$ for $1\leq j < l \leq k.$ Denote the clique complex of the adjacency graph of $\mathcal{O}$ by $\Delta\paren{O}.$  

\begin{Proposition}\label{prop:nerve}
$O$ is homotopy equivalent to $\Delta\paren{O}.$    
\end{Proposition}
\begin{proof}
We will show that $\Delta\paren{O}$ coincides with another abstract simplicial complex called the nerve of $\mathcal{O}.$ The \emph{nerve} of a finite collection of sets $\mathcal{X}=\set{X_1,\ldots,X_n}$ is the simplicial complex $\mathrm{Nerve}\paren{\mathcal{X}}$ with vertex set $\set{X_1,\ldots,X_n}$  so that 
\[\paren{X_{i_1},\ldots,X_{i_l}}\in \mathrm{Nerve}\paren{\mathcal{X}}\iff \bigcap_{j=1}^l X_{i_{j}}\neq \varnothing\,.\]
The nerve theorem --- a standard result of combinatorial topology (see  Section 3 of~\cite{bauer2023unified} for a relatively elementary proof) --- states that if all nonempty intersections of subsets of $\mathcal{X}$ are convex then  $\mathrm{Nerve}\paren{\mathcal{X}}$ is homotopy equivalent to the union $\cup_{X\in \mathcal {X}} X.$ This applies to $\mathcal{O},$ as the intersection of a collection of permutohedra of $A_{d}^*$ is either empty or is a shared face of the permutohedra. 

It suffices to show that  $\mathrm{Nerve}\paren{\mathcal{X}}=\Delta\paren{O}.$ This is true because the intersection of a collection of $k$ permutohedra of $A_{d}^*$ is non-empty if and only if each pair of permutohedra in the collection shares a $(d-1)$-face. If this holds, the intersection is then $(d-k+1)$-dimensional. In other words, the simplicial complexes $\mathrm{Nerve}\paren{\mathcal{O}}$ and $\Delta\paren{O}$ coincide, and the claim follows. For a proof of this fact, see Section 3 of~\cite{choudhary2019polynomial}.  Note that the corresponding statement for $\Z^d$ is false, as two cubes can meet at a vertex.
\end{proof}

\begin{figure}[t]
    \centering
    \includegraphics[width=.4\textwidth]{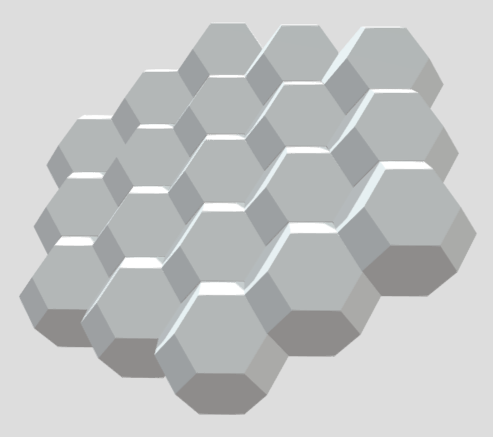}
    \quad
    \includegraphics[width=.4\textwidth]{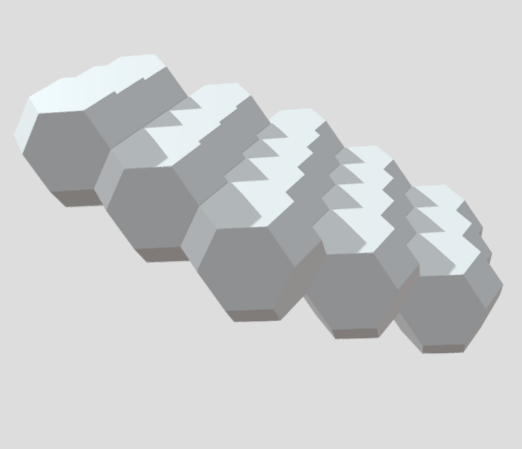}
    \caption{\label{fig:permmonotonicity3} A portion of $S$ in 3 dimensions, shown from two different angles.}
\end{figure}

We now construct a subset of $A_d^*$ with the same pairwise adjacencies  as $A_{d-1}^*.$
Recall that 
$$\hat{\R}^d \coloneqq \set{(x_0,x_1,\ldots,x_d)\in \R^{d+1} : \sum_{k=0}^d x_k = 0}.$$ 

Let
\[\vec{v}_1=\paren{\frac{-d}{d+1},\frac{1}{d+1},\ldots,\frac{1}{d+1}},\ldots,\vec{v}_{d+1}=\paren{\frac{1}{d+1},\ldots,\frac{1}{d+1},\frac{-d}{d+1}}\]
where the vectors have $(d+1)$ coordinates. A $\mathbb{Z}$-basis for $A_{d}^*$ is given by the vectors $\vec{v}_1,\ldots,\vec{v}_d,$ where we note that
\begin{equation}\label{eq:permutobasis}
    \vec{v}_{d+1}=-\paren{\vec{v}_1+\ldots+\vec{v}_{d}}\,.
\end{equation}

We require one more fact about the combinatorics of the permutohedral lattice.

\begin{Lemma}\label{lemma:adjacency}
The permutohedron centered at the origin of $A_d^*$ is adjacent to exactly the permutohedra centered at the points
\begin{equation}
    \label{eq:delaunayvertices}
\mathcal{V}_1=\set{\paren{a_1\vec{v_1}+\ldots+a_{d+1}\vec{v}_{d+1}}: \paren{a_1,\ldots,a_{d+1}\in \set{0,1}^{d+1}\setminus \set{\vec{0},\vec{1}}}}
\end{equation}
and all of these points are distinct. Here $\vec{1}=\paren{1,\ldots,1}.$ 
\end{Lemma}
\begin{proof}
This follows from Theorem 2.5 of~\cite{baek2009some}. Let $\theta$ be the permutohedron centered at the origin of $A_d^*,$ and let $\mathcal{V}_2$ be the set of the points at the centers of permutohedra adjacent to $\theta.$ Recall that the Delaunay triangulation of a set of points in general position is the dual complex of its Voronoi diagram. $A_{d}^*$ is in general position (this is Lemma 3.4 of~\cite{choudhary2019polynomial}), so the permutohedra adjacent to $\theta$ are those centered at points in the $1$-neighborhood of $\vec{0}$ in the Delaunay triangulation. These are exactly the vertices which are contained in top-dimensional cells adjacent to the origin.

Theorem 2.5 of~\cite{baek2009some} states that all top-dimensional Delaunay cells may be obtained from the simplex with vertices
\begin{equation}
    \label{eq:delaunay2}
\vec{0},\vec{v}_1,\vec{v}_1+\vec{v}_2,\ldots,\vec{v}_1+\ldots+\vec{v}_{d}
\end{equation}
by symmetries of $A_{d}^*.$ As such, $\mathcal{V}_2\cup\set{\vec{0}}$ is the set of vertices obtained from \ref{eq:delaunay2} by applying symmetries which map one of the vertices to $\vec{0}.$ 

The symmetries of $A_d^*$ are translations and permutations of the coordinates. Applying permutations yields that $\mathcal{V}_1\subset \mathcal{V}_2.$ We now verify that all vertices in $\mathcal{V}_2$ can be obtained only by a permutation of the coordinates. Consider a translated Delaunay cell containing $\vec{0},$ which without loss of generality is of the form 
\[-\vec{v}_1-\ldots -\vec{v}_k,-\vec{v}_2- \ldots -\vec{v}_k,\ldots,\vec{0},\vec{v}_{k+1},\vec{v}_{k+1}+\vec{v}_{k+2},\ldots,\vec{v}_{k+1}+\ldots+\vec{v}_{d}\,.\]
Then by (\ref{eq:permutobasis}), we can write this as 
\begin{align*}
&\vec{v}_{k+1}+\ldots+\vec{v}_{d} + \vec{v}_{d+1}, \ldots,\vec{v}_{k+1}+\ldots+ \vec{v}_{d+1} + \vec{v}_1 + \ldots+\vec{v}_{k-1},\vec{0},\\
&\vec{v}_{k+1},\vec{v}_{k+1}+\vec{v}_{k+2},\ldots,\vec{v}_{k+1}+\ldots+\vec{v}_{d}\,,
\end{align*}
which is obtained by a permutation. Thus $\mathcal{V}_2\subset \mathcal{V}_1.$ Finally, distinctness of the points in Equation~\ref{eq:delaunayvertices} follows from the fact that the number of $(d-1)$-dimensional faces of the $d$-dimensional permutohedron is $2^{d+1}-2.$
\end{proof}

Let $L^{\perp}$ be the orthogonal complement of the line spanned by $\vec{v}_{d+1}$ in $\hat{R}^d.$ 
Notice that
\[L^{\perp}=\set{\vec{v}\in \hat{\R}^d:\vec{v} \perp \paren{1,\ldots,1,0}, \vec{v} \perp \vec{e}_{d+1}}\]
so we can identify this space with $\hat{\R}^{d-1}.$ Then $A_{d-1}^*$ is spanned by the vectors $\vec{w}_1,\ldots,\vec{w}_{d-1}$ where
\[\vec{w}_1=\paren{\frac{-\paren{d+1}}{d},\frac{1}{d},\ldots,\frac{1}{d},0},\ldots,\vec{w}_{d}=\paren{\frac{1}{d},\ldots,\frac{1}{d},\frac{-\paren{d+1}}{d},0}\]
Let $\pi:\hat{\R}^{d}\to \hat{\R}^{d-1}$ be the orthogonal projection, and let $S$ be the sublattice of $A_{d}^*$ generated by  $\vec{v}_1,\ldots,\vec{v}_{d-1}.$ That is, $S$ is the intersection of $A_{d}^*$ with the hyperplane $x_d=x_{d+1}.$  
\begin{Lemma}
$\pi$ maps $S$ isometrically to $A_{d-1}^*.$ Moreover, the $d$-dimensional permutohedra centered at $\vec{u}_1,\vec{u}_2\in S$ are adjacent if and only if the $(d-1)$-dimensional permutohedra centered at $\pi\paren{\vec{u}_1},\pi\paren{\vec{u}_2}$ are.   
\end{Lemma}
\begin{proof}
We have that
\[\vec{v}_i\cdot \vec{v}_j=\begin{cases} \frac{d}{d+1} & i=j \\ -\frac{1}{d+1} & i\neq j \end{cases}\]
so 
\begin{align*}
\pi\paren{\vec{v}_{1}}
&=\vec{v}_1 - \mathrm{proj}_{\vec{v}_{d+1}} \vec{v}_1\\ 
&= \vec{v}_1 - \frac{\vec{v}_1 \cdot \vec{v}_{d+1}}{\vec{v}_{d+1} \cdot \vec{v}_{d+1}} \vec{v}_{d+1}\\
&= \vec{v}_1 + \frac{1}{d} \vec{v}_{d+1}\\
&=\paren{-\frac{d-1}{d},\frac{1}{d},\ldots,\frac{1}{d},0}\\
&= \vec{w}_1\,,
\end{align*}
and a similar computation shows that  $\pi\paren{\vec{v}_i}=\vec{w}_i$ for $i=1,\ldots, d.$ This implies the first claim.  

The second claim follows from Lemma~\ref{lemma:adjacency}, applied to both $\vec{u}_1,\vec{u}_2\in S$ and $\pi\paren{\vec{u}_1},\pi\paren{\vec{u}_2}\in A_{d-1}^*.$  
\end{proof}

\begin{Proposition}\label{prop:permuto_sublattice}
Let $\theta^d$ be the permutohedron centered at the origin of $A_d^{*},$ and let $F:A_{d-1}^*\to S$ be the inverse of $\pi\mid_{S}.$ 
For any subset $E \subset \mathcal{A}_{d-1}^*,$
\[\bigcup_{\vec{w} \in E} \paren{\theta^{d-1}+\vec{w}} \simeq \bigcup_{\vec{w} \in E}\paren{\theta^d+F\paren{\vec{w}}}\,,\]
where $\simeq$ denotes homotopy equivalence.
\end{Proposition}

\begin{proof}
By the previous lemma, the two sets of permutohedra have the same adjacency graph. Thus, it follows from Proposition~\ref{prop:nerve} that their unions are homotopy equivalent.
\end{proof}
We now show that we can ``replace'' a permuotohedron of $S$ with a certain set of permutohedra outside of $S$ in a way that preserves giant cycles. Let $\theta$ be a permutohedron centered at a point of $S,$ which we assume without loss of generality is the origin. We partition the neighbors of $\theta$ into three disjoint sets. Let
\[U\paren{\theta} = \bigcup_{\paren{a_1,\ldots,a_{d-1}}\in \set{0,1}^d} \paren{\theta + a_1\vec{v}_1+\ldots + a_{d-1}\vec{v}_{d-1} + \vec{v}_{d+1}}\,\]
be the neighbors of $\theta$ contained in the translate $S + \vec{v}_{d+1}$, let 
\[V\paren{\theta} = \bigcup_{\paren{a_1,\ldots,a_{d}}\in \set{0,1}^d\setminus \set{\vec{0},\vec{1}}} \paren{\theta + a_1\vec{v}_1+\ldots + a_{d-1}\vec{v}_{d-1}+a_d\vec{v}_d+a_d\vec{v}_{d+1}}\]
be those neighbors which are contained in $S,$ and let
\[W\paren{\theta} = \bigcup_{\paren{a_1,\ldots,a_{d-1}}\in \set{0,1}^d} \paren{\theta + a_1\vec{v}_1+\ldots + a_{d-1}\vec{v}_{d-1} + \vec{v}_{d}}\]
be the neighbors contained in $S-\vec{v}_{d+1}.$ Observe that these three sets consist of the neighbors of $\theta$ centered at points satisfying $x_{d+1}>x_d,$ $x_{d+1}=x_d,$ and $x_{d+1}<x_d,$ respectively. 

Let $X_1\coloneqq \partial\theta \cap U\paren{\theta}$ and $X_2 \coloneqq \partial\theta \cap W\paren{\theta}.$ We now show that faces of $\theta$ are either contained in $X_1$ or $X_2$ or have a subface that is.

\begin{Lemma}\label{lemma:freeface}
    Let $1\leq i\leq d-1$ and let $\tau$ be an $i$-dimension face of $\tau$ that is neither contained in $X_1$ nor $X_2.$ Then $\tau$ is orthogonal to the hyperplane $x_d=x_{d+1}.$ In addition, it has an $(i-1)$-dimensional face contained in $X_2.$ 
\end{Lemma}

\begin{proof}
We begin by noting two preliminary facts regarding a general face $\sigma$ of $\theta.$ Suppose that $\sigma$ is the intersection of permutohedra centered at points $p_1,\ldots,p_k$ of $A_d^*.$ Fact 1: if not all of the points $p_1,\ldots,p_k$ are contained in $S,$ then $\sigma$ is a face of $X_1$ or $X_2.$ This is true because $\set{U\paren{\theta},V\paren{\theta},W\paren{\theta}}$ is a partition of the neighbors of $\theta.$ Fact 2: as the permutohedra are Voronoi cells, $\sigma$ is contained in the intersection of the hyperplane bisectors for each pair of points in $\set{p_1,\ldots,p_k}.$ In particular, if $\set{p_1,\ldots,p_k}\subset S$ then each of these bisectors is orthogonal to the hyperplane $x_{d}=x_{d+1}$ and $\sigma$ is orthogonal to that hyperplane, as well.   

Now suppose that $\tau$ is the intersection of the permutohedra $\theta_1=\theta,\ldots,\theta_{d-i+1}$ centered at points $p_1,\ldots,p_{d-i+1}$ of $A_d^*.$ By Fact 1 above, $p_1,\ldots,p_{d-i+1}$ are in $S.$ It then follows from Fact 2 that $\tau$ is orthogonal to the hyperplane $x_d=x_{d+1}.$ As $\tau$ is bounded, it contains a face $\tau'$ that is not orthogonal to $x_d=x_{d+1}.$ A second application of Fact 2 yields that $\tau'$ is the intersection of permutohedra, at least one of which is not contained in $S.$ By Fact 1, $\tau'$ is a face of either $X_1$ or $X_2.$
\end{proof}

In the next two lemmas, we use a standard construction from combinatorial algebraic topology: if $A$ is a polyhedral complex, $\sigma$ is a cell of $A$ not contained in any higher-dimensional cells, $\tau\subset \sigma$ is a face of $\sigma$ with $\mathrm{dim}\tau=\mathrm{dim}\sigma-1,$ and $\sigma$ is the only higher-dimensional neighbor of $\tau$, then $A$ deformation retracts to the polyhedral complex $A\setminus \set{\sigma,\tau} .$ In this context, $\tau$ is called a \emph{free face} and the deformation retraction is called an \emph{elementary collapse}. This notion is usually defined for simplicial complexes (see Definition 6.13 of~\cite{kozlov2008combinatorial}) but is also valid for more general polyhedral complexes (an even more general notion is given in Definition 11.12 of~\cite{kozlov2008combinatorial}). We now consider the topological structure of $X_1$ and $X_2.$

\begin{Lemma}\label{lemma:x1ball}
    $X_1$ and $X_2$ are homeomorphic to $(d-1)$-dimensional balls. 
\end{Lemma}

\begin{proof}
$X_1$ and $X_2$ are symmetric so we may consider $X_2$ alone. We will show that $\theta \cap \set{x_d\geq x_{d+1}}$ deformation retracts to $X_2.$ This will suffice, as the former set is a $(d-1)$-dimensional ball (it is the bisection of the boundary of a convex polyhedron by a hyperplane).

As a preliminary step, we note that $X_1$ is contained in $\set{x_d< x_{d+1}}.$ First, one can verify directly from Lemma~\ref{lemma:adjacency} that if $\theta_1\in U\paren{\theta}$ and $\theta_2\in W\paren{\theta}$ then $\theta_1$ and $\theta_2$ are not adjacent and $\theta_1\cap\theta_2=\varnothing.$ Thus $X_1\cap X_2 = \varnothing.$ This implies that $X_1$ does not intersect $\set{x_d=x_{d+1}},$ because $X_1$ and $X_2$ are symmetric about that hyperplane.

Let $\theta'$ be the polyhedron obtained from $\theta$ by subdividing each $i$-dimensional cell $\sigma$ that intersects the plane $x_d=x_{d+1}$ into two $i$-dimensional cells $\sigma^+=\sigma\cap \set{x_d\geq x_{d+1}}$ and $\sigma^-=\sigma\cap \set{x_d\leq x_{d+1}},$ and adding a new $(i-1)$-dimensional cell $\sigma^0=\sigma\cap \set{x_d=x_{d+1}}.$ It is easily checked that this produces a well-defined polyhedron (note that all faces of $\theta$ that intersect this hyperplane are orthogonal to it). We deform $Y=\theta'\cap  \set{x_d\geq x_{d+1}}$ to $X_2$ by elementary collapses, in order by decreasing dimension. Suppose that we have removed all $(i+1)$-dimensional faces of $Y\setminus X_2,$ and that $\tau$ is an $i$-dimensional face of $Y \setminus X_2.$ Since $X_1 \subset \set{x_d < x_{d+1}},$ it follows from Lemma~\ref{lemma:freeface} that $\tau$ must be orthogonal to $x_d\geq x_{d+1}.$ Then, by construction,  $\tau$ is one of the faces $\sigma^-$ described above, and it has a free face $\sigma^0.$ Thus, we may remove $\sigma^-$ and $\sigma^0$ by an elementary collapse.      
\end{proof}

Now we are ready to construct a deformation retraction that allows us to replace $\theta$ with $U\paren{\theta}.$ 

\begin{Lemma}\label{lemma:permutocollapse}
    Let $Q=\bigcup_{\vec{w} \in E}\paren{\theta^d+F\paren{\vec{w}}}$ for some $E \subset S.$ There is a deformation retraction from $Q \cup \theta \cup U\paren{\theta}$ to $Q \cup U\paren{\theta} \setminus \theta.$
\end{Lemma}

\begin{proof} 
Let $\theta' \in V\paren{\theta}$ and let $p= a_1\vec{v}_1+\ldots + a_{d-1}\vec{v}_{d-1}+a_d\vec{v}_d+a_d\vec{v}_{d+1}$ be the point at its center, for $\paren{a_1,\ldots,a_{d}}\in \set{0,1}^d\setminus \set{\vec{0},\vec{1}}.$ If $a_d=0$ then, by Lemma~\ref{lemma:adjacency}, both $\theta'$ and $\theta$ are adjacent to the permutohedron of $W\paren{\theta}$ centered at $p+\vec{v}_d.$ Otherwise, they both neighbor the permutohedron of $W\paren{\theta}$ centered at $p-\vec{v}_{d-1}.$  It follows that every $(d-1)$-face in $\theta \cap V\paren{\theta}$ contains a $(d-2)$-face in $X_2.$ 

Let $f$ be a $(d-1)$-dimensional face of $X_2.$ Then $f$ is a face of two permutohedra of $A_d^*$: $\theta$ and another permutohedron contained in $W\paren{\theta}.$ Since $Q \cup \theta$ and $W\paren{\theta}$ share no $d$-cells, $Q \cup \theta$ has a free face $f$ and thus deformation retracts to  $Q\setminus \theta.$ Then since $X_2$ is homeomorphic to a $(d-1)$-dimensional ball by Lemma~\ref{lemma:x1ball},  $\partial \theta \setminus f$ deformation retracts to $\partial \theta \setminus X_2.$ 

We now iteratively collapse the faces of $Y=\theta \setminus \paren{Q \cup U\paren{\theta}},$ from dimension $k=d-1$ down to dimension $k=1.$ Assume that we have already removed all faces of $Y$ of dimension at least $k+1$ and let $\tau$ be a $k$-face of the partially collapsed complex. By Lemma~\ref{lemma:freeface}, $\tau$ has a $(k-1)$-subface $\tau'$ contained in $X_2.$ Then $\tau'=\tau\cap \theta'$ where $\theta'$ is a permutohedron on $W\paren{\theta}.$ As such, all $k$-faces containing $\tau'$ --- with the exception of $\tau$ --- are faces of $\theta'$ and therefore are faces of $X_2.$ It follows that $\tau'$ is a free face of $\tau.$ Therefore, we may remove $\tau$ and $\tau'$ via an elementary collapse. Once all of the collapses are made, we are then left with $Q \cup U\paren{\theta} \setminus \theta$ as desired.

\end{proof}

We are now ready to prove strict monotonicity.

\begin{Proposition}
$p^{\hexagon}_i\paren{d} < p^{\hexagon}_i\paren{d-1}< p^{\hexagon}_{i+1}\paren{d}$
\label{prop:monotonicity_permuto}
\end{Proposition}

\begin{proof}
Let $\mathbb{S}$ and $\mathbb{S}'$ be the sets of permuotohedra centered at points of $S$ and $S\cup\paren {S+\vec{v}_{d+1}},$ respectively.

The inequalities $p^{\hexagon}_i\paren{d} \leq p^{\hexagon}_i\paren{d-1} \leq p^{\hexagon}_{i+1}\paren{d}$ follow from Proposition~\ref{prop:permuto_sublattice}, via an argument identical that that for  Proposition~\ref{prop:monotonicity}. The proofs of the strict inequalities are also similar to the case of plaquette percolation. Here, $\mathbb{S}$ and $\mathbb{S}'$ play the roles of $T$ and $T',$ respectively.

By Lemma~\ref{lemma:permutocollapse}, $Q \cup \theta \cup U\paren{\theta}$ deformation retracts to $Q \cup U\paren{\theta} \setminus \theta,$ so if there is a giant cycle in $Q \cup \theta,$ there is also a giant cycle in $Q \cup U\paren{\theta}.$ The rest of the argument is nearly identical to that for Proposition~\ref{prop:monotonicity}: site percolation with probability $p$ on $\mathbb{S}'$ can be coupled with site percolation on $\mathbb{S}$ with a slightly higher parameter by sometimes including $\theta$ when all permutohedra in $U\paren{\theta}$ are present. There is overlap between the sets $U\paren{\theta},$ but this is dealt with in a similar manner as in the plaquette construction. This argument yields the strict inequality  $p^{\hexagon}_i\paren{d} < p^{\hexagon}_i\paren{d-1}.$ We also obtain the strict inequality $p^{\hexagon}_i\paren{d-1} < p^{\hexagon}_{i+1}\paren{d}$ by the same duality argument as before.

\end{proof}

\section{Future directions}

It seems that not much is known about percolation with higher-dimensional cells or homological analogues of bond or site percolation.

\begin{itemize}

\item Do $\lambda^{\square}\paren{N,i,d}$ and $\lambda_i^{\hexagon}\paren{N,d}$ converge as $N\to \infty$ for all $i,d?$

\item Can the minor restrictions on the characteristic of the coefficient field of the homology be removed?

\item Are there scaling limits for plaquette percolation? For bond percolation in the plane at criticality, conjecturally we get SLE. This could be a reasonable question to approach experimentally. 

\item Is there a limiting distribution for $\rank\phi_*$ at criticality, as $N \to \infty$? When $d=2i$, our results imply that the distribution is symmetric and the expectation satisfies $\mathbb{E}[ \rank \phi_*] = \binom{d}{d/2}/2$, but at the moment we do not know anything else.

\item  One of the most interesting possibilities we can imagine would be a generalization of the Harris--Kesten theorem when $d=2i$, on the whole lattice $\Z^d$ rather than on the torus $\T^d_N$. One possibility might be to compactify $\R^d$ to a torus $T^d$. In various proofs of the Harris--Kesten theorem, a key step is to go from crossing squares to crossing long, skinny rectangles---see, for example, Chapter 3 of~\cite{BR06}. One difficulty is that we do not currently have a high-dimensional version of the Russo--Seymour--Welsh method, passing from homological ``crossings'' of high-dimensional cubes to long, skinny boxes.
\end{itemize}

\subsubsection*{Acknowledgements}
We thank Omer Bobrowski, Russell Lyons, Primoz Skraba, and the anonymous referee for comments on an earlier version 
of the manuscript. Also, M.K.\ thanks Christopher Hoffman for many helpful conversations.

\bibliographystyle{alpha} 
\bibliography{bibliography}

\end{document}